\documentclass[11pt]{article}

\usepackage[T1]{fontenc}
\usepackage{lmodern}
\usepackage[margin=1in]{geometry}
\usepackage{amsmath,amssymb,amsthm,mathtools}
\usepackage{enumitem}
\usepackage[hidelinks]{hyperref}
\usepackage{microtype}

\newtheorem{theorem}{Theorem}
\newtheorem{lemma}[theorem]{Lemma}
\newtheorem{proposition}[theorem]{Proposition}
\theoremstyle{remark}
\newtheorem{remark}[theorem]{Remark}
\newtheorem{corollary}[theorem]{Corollary}

\newcommand{\B}{\mathbb B}
\newcommand{\C}{\mathbb C}
\newcommand{\R}{\mathbb R}
\newcommand{\Z}{\mathbb Z}
\newcommand{\Q}{\mathbb Q}
\newcommand{\one}{\boldsymbol 1}
\newcommand{\dd}{\,\mathrm d}
\newcommand{\Ric}{\operatorname{Ric}}
\newcommand{\Rec}{\operatorname{recc}}
\newcommand{\Log}{\operatorname{Log}}
\newcommand{\Vol}{\operatorname{Vol}}

\date{}

\begin{document}

\title{Yau's conjecture on smooth Reinhardt domains}

\author{Yuan Yuan \medskip \\
Institute for Theoretical Sciences, Westlake University,\\
Hangzhou 310024, Zhejiang, China\\
\texttt{yuanyuan@westlake.edu.cn}}

\maketitle

\begin{abstract}
We prove that a possibly unbounded Reinhardt domain in $\mathbb C^n$
with smooth boundary and a complete Bergman-Einstein  metric is
biholomorphic to the complex unit ball.  In particular, no unbounded
smooth Reinhardt domain admits a complete Bergman-Einstein  metric.
\end{abstract}

\section{Introduction}

The Bergman metric is one of the principal biholomorphic invariants of a
complex domain.  It is constructed directly from square-integrable
holomorphic functions and therefore connects complex analysis with
K\"ahler geometry.  On the unit ball, the Bergman metric is complete and
K\"ahler--Einstein.  This naturally leads to the rigidity question of
whether the Einstein condition characterizes the ball, or at least a
homogeneous domain, among suitable pseudoconvex domains.

Several forms of this question have appeared in the literature.
Lu's uniformization theorem \cite{L66} shows that a domain carrying a
complete Bergman metric of constant holomorphic sectional curvature is
biholomorphic to the ball.  In 1979, Cheng \cite{C79} asked whether the
weaker K\"ahler--Einstein condition already characterizes the ball among
smoothly bounded strongly pseudoconvex domains.  Yau \cite{Y82}
subsequently placed this problem in the broader setting of bounded
pseudoconvex domains and proposed that a complete Bergman-Einstein
metric should force the domain to be homogeneous.  These questions are
usually referred to as Cheng's conjecture and  Yau's
conjecture.

The first substantial progress used boundary asymptotics of the
Bergman kernel.  Fu and Wong \cite{FW97} related the Einstein equation
to the vanishing of the logarithmic obstruction in the Bergman
kernel expansion. Cheng’s conjecture in the two dimensional strongly pseudoconvex domain is proved by Fu-Wong \cite{FW97} and Nemirovski-Shafikov  \cite{NS06}.  
Huang and Xiao
\cite{HX20,HX21} later resolved Cheng's conjecture in all dimensions.
The same circle of ideas was subsequently developed for Stein spaces,
spherical boundaries, and finite ball quotients (cf. \cite{EXX22,HL23,GS26,GSa26}).  Beyond strong pseudoconvexity,
Savale and Xiao \cite{SX25} treated smooth finite-type domains in
complex dimension two.  Related rigidity and nonexistence phenomena
for nonsmooth boundary models were established in \cite{HJL25}.
For Reinhardt domains, the interaction between K\"ahler--Einstein
metrics and logarithmic coordinates had already been investigated by
Isaev \cite{I95} and Yau's conjecture is shown to be true on Reinhardt domains of finite type in $\mathbb{C}^2$ \cite{FW97}.

The decisive breakthrough toward Yau's conjecture was achieved and Yau's conjecture in all dimensions for bounded pseudoconvex domains with real analytic boundary
 was solved by
Hsiao, Huang, and Li \cite{HHL26}. Their deep paper 
\cite{HHL26} provides the fundamental
analytic input for the present paper.  Their localization theorem makes
precise boundary analysis of the Bergman kernel available near
strongly pseudoconvex boundary points even when the ambient domain is unbounded.  Two
consequences are especially important here.  First, if a possibly
unbounded pseudoconvex domain has a Bergman-Einstein metric, then its
boundary is locally spherical near every smooth strongly pseudoconvex
point.  Second, the Bergman invariant function
$ J_\Omega$ 
assumes the universal value
$ J_\Omega=\frac{(n+1)^n\pi^n}{n!}$
as soon as such a boundary point exists.  In particular, the Einstein
constant is normalized by
\(\operatorname{Ric}(\omega_\Omega)=-\omega_\Omega\).
Their global theorem for bounded pseudoconvex domains with
real-analytic boundary is also used in the bounded case.
In fact, the work of Hsiao-Huang-Li  \cite{HHL26} supplies the passage from local boundary geometry to
Bergman rigidity, including in the unbounded setting.  The purpose of
the present paper is to combine this analytic principle with the
convex and arithmetic structure imposed by torus symmetry.  Our
main result is as follows.

\begin{theorem}\label{thm:main}
Let $\Omega\subset\C^n$ be a (possibly unbounded) Reinhardt domain with
$C^\infty$-smooth boundary.  If the Bergman metric of $\Omega$ is well
defined and is a complete Einstein metric, then $\Omega$ is biholomorphic
to the complex unit ball $\B^n$.
\end{theorem}

We always write
$ K_\Omega(z)=K_\Omega(z,z)$ for the diagonal Bergman kernel on $\Omega$,  $\omega_\Omega=\sqrt{-1}\partial \bar\partial \log K_\Omega$ for the Bergman metric and $ J_\Omega =  \frac{\det\bigl(
 \partial_j\partial_{\bar k}\log K_\Omega
 \bigr)} {K_\Omega}$ for the Bergman invariant function. 
The Bergman metric is called well defined when
\(K_\Omega(z,z)>0\) and the Hermitian matrix
$ \left( \partial_j\partial_{\bar k}\log K_\Omega(z,z) \right)_{j,k=1}^n$
is positive definite throughout \(\Omega\). 

We briefly describe the proof.  Completeness of the Bergman metric
implies pseudoconvexity, and hence the logarithmic image
$ L= \left\{ (\log|z_1|,\ldots,\log|z_n|): z\in\Omega\cap(\C^*)^n \right\}$
is convex.  We first use the geometry of \(L_\Omega\) to produce a
strongly pseudoconvex boundary point whose coordinates are all
nonzero.  The exponential covering transfers this point to the tube
hypersurface
$ \partial L_\Omega+i\R^n.$
The local sphericality theorem of Hsiao--Huang--Li then yields a
spherical open piece.  Results of Isaev \cite{I93,I11} extend this
local information across the connected logarithmic boundary.  The
resulting strongly pseudoconvex spherical tube  therefore belongs to one of the
three types in the Dadok--Yang classification \cite{DY85}.

Two additional arguments are needed to turn this classification into
a classification of the original Reinhardt domain.  If
\(0\in\Omega\), the universal identity for \(J_\Omega\), together with
a sharp moment inequality, forces \(\Omega\) to be an 
ellipsoid.  For bounded domains, the remaining Dadok--Yang models are
eliminated or reduced to the real-analytic theorem of
Hsiao--Huang--Li.  
For unbounded domains, we combine torus symmetry with the geometry
of the logarithmic image and the structure of the Bergman kernel to
show that none of the three models can occur.
The one-dimensional case follows from Lu's theorem and the
elementary classification of Reinhardt domains in \(\C\).

\medskip

We now prove Theorem \ref{thm:main} assuming Theorem \ref{bounded} and Theorem \ref{unbounded}. These two theorems are proved in later sections. 

\begin{proof}[Proof of Theorem \ref{thm:main}]
By Theorem \ref{bounded} and Theorem \ref{unbounded}, it suffices to consider the one-dimensional case.  If $\Omega$ is bounded, the Bergman-Einstein metric has
constant curvature.  Lu's theorem \cite{L66} shows that a bounded domain
with complete Bergman metric of constant curvature is biholomorphic to
the unit disk.
Suppose that $\Omega$ is unbounded.  Then $\Omega$ can only be $\C$ or $ \Omega_R=\{z\in\C:|z|>R\}
$ with $R \geq 0$. It is easy to verify that either the Bergman metric does not exist or the Bergman metric is incomplete. 
\end{proof}

The paper is organized accordingly.  After recalling the logarithmic
geometry of pseudoconvex Reinhardt domains and the Dadok--Yang
classification, we prove the sharp moment inequality and treat domains
containing the origin.  The bounded and unbounded cases are then
considered separately.

\section{Preliminaries}

We first collect some standard results on Reinhardt domains; the reader may
refer to \cite{JP08} for more detailed discussions. 

Let $\Omega \subset \C^n$ be a  Reinhardt domain. 
Let
\[
 L=\Log \Omega
 =\{(\log|z_1|,\dots,\log|z_n|):
               z\in\Omega\cap(\C^*)^n\}
\]
be the logarithmic image of $\Omega$.
When $\Omega$ is pseudoconvex, $L$ is a geometrically convex open set in $\R^n$ and relatively complete (cf. Theorem 1.11.13 in \cite{JP08}).
Let $S=\partial L$ and $M:=S+i\R^n$ be a tube over $S$.  The exponential map
\[
\operatorname{Exp}\colon \mathbb C^n\rightarrow(\mathbb C^*)^n,
\quad
\operatorname{Exp}(w)
   =(e^{w_1},\ldots,e^{w_n}),
\]
is a local biholomorphism and satisfies $\operatorname{Exp}^{-1} \bigl(\partial \Omega \cap(\mathbb C^*)^n\bigr) =M.$ 
The pull back of a torus-invariant smooth defining function for
\(\partial \Omega\) is a defining function of the form
\(r(\operatorname{Re}w)\), with \(dr\neq0\) on \(S\). 

The following lemma is well-known to experts (cf. \cite{FIK96a, IK98}).
In fact, it suffices to assume that
$\Omega$ is a $C^1$-smooth Reinhardt domain.

\begin{lemma}\label{lem:origin-not-boundary}
Let $\Omega \subset\C^n$ be a Reinhardt domain with $C^1$-smooth boundary.  Then
$0\notin\partial \Omega$.
\end{lemma}

\begin{proof}
In a tubular neighborhood of the boundary, let $\rho$ be the signed
Euclidean distance to $\partial \Omega$.  Since $\Omega$ is Reinhardt, $\rho$ is
invariant under the torus action of $(S^1)^n$.  If
$0\in\partial \Omega$, then $d\rho(0)$ is fixed by every coordinatewise
rotation.  The only real covector at the origin fixed by this action is
the zero covector.  Thus $d\rho(0)=0$, contradicting the fact that $\rho$
is a defining function.
\end{proof}

We also record the following standard consequence of the description of
square-integrable Laurent monomials on pseudoconvex Reinhardt domains
(cf.  Lemma 5 in \cite{Z99}, \cite{J12}).

\begin{lemma}\label{lem:no-line}
Let $\Omega$ be a pseudoconvex Reinhardt domain. 
If $L$ contains an affine line, then the Bergman space $A^2(\Omega)$ is trivial.
\end{lemma}

The following fundamental classification result due to Dadok and Yang \cite{DY85} (cf. also
Theorem 5.1 in \cite{I11}) is crucial in our proof.

\begin{theorem}\label{DY}
Let $M=S+i\mathbb{R}^n\subset\mathbb{C}^n$, $ n\geq2$, 
be a connected closed strongly pseudoconvex spherical tube
hypersurface. Two tube hypersurfaces are called affinely equivalent
if their bases are related by a real affine transformation
$ x\mapsto Ax+b$ for some $ A\in\operatorname{GL}(n,\mathbb{R})$ and $ b\in\mathbb{R}^n.$
Write $x=(x_0,x_1,\ldots,x_{n-1})\in\mathbb{R}^n.$
Then \(M\) is affinely equivalent to exactly one of the tube
hypersurfaces whose bases are given by the following equations:
\begin{enumerate}[label=\textup{Type \Roman*:},leftmargin=5em]

\item
\[
x_0
=
\sum_{\alpha=1}^{m}e^{x_\alpha}
+
\sum_{\alpha=m+1}^{n-1}x_\alpha^2,
\quad
0\leq m\leq n-1.
\]

\item
\[
\sin x_0
=
\sum_{\alpha=1}^{n-1}e^{x_\alpha},
\quad
0<x_0<\pi.
\]

\item
\[
\sum_{\alpha=0}^{n-1}e^{x_\alpha}=1.
\]
\end{enumerate}
These  \(n+2\) tube hypersurfaces are mutually affine inequivalent.
\end{theorem}

\begin{remark}\label{DY2}
In the setting \(S=\partial L\), where \(L\subset\mathbb{R}^n\) is
convex, the corresponding convex components are as follows:
\begin{align*}
\textup{Type I:}\quad
L_{\mathrm I}^{(m)}
&=
\left\{
x\in\mathbb{R}^n:
x_0>
\sum_{\alpha=1}^{m}e^{x_\alpha}
+
\sum_{\alpha=m+1}^{n-1}x_\alpha^2
\right\},
\\[4pt]
\textup{Type II:}\quad
L_{\mathrm{II}}
&=
\left\{
x\in\mathbb{R}^n:
0<x_0<\pi,\quad
\sum_{\alpha=1}^{n-1}e^{x_\alpha}<\sin x_0
\right\},
\\[4pt]
\textup{Type III:}\quad
L_{\mathrm{III}}
&=
\left\{
x\in\mathbb{R}^n:
\sum_{\alpha=0}^{n-1}e^{x_\alpha}<1
\right\}.
\end{align*}
\end{remark}

\bigskip

The recession cone $\operatorname{recc}(C)$ of a nonempty convex set
$C\subset\R^n$ is given by
\[
\operatorname{recc}(C)
=
\left\{
d \in \mathbb{R}^n
\;\middle|\;
x + td \in C
\quad \text{for all } x \in C,\; t \ge 0
\right\},
\]
and $d$ is called a recession direction. 
If \(C\) is also closed, it suffices to check the condition for any single \(x_0\in C\). Namely, 
\[
\operatorname{recc}(C)
=
\left\{d:x_0+t d\in C\ \text{for all }t\ge0\right\}.
\]
The recession cone is always a convex cone and contains \(0\). A nonempty
closed convex set is bounded exactly when
$ \operatorname{recc}(C)=\{0\}.$

\begin{lemma}\label{lem:recession}
Let  $\Omega\subset\C^n$ be a pseudoconvex Reinhardt domain with $C^1$-smooth boundary. 
Let $ I=\{j:\Omega \cap\{z_j=0\}\neq\varnothing\}.$ Denote $e_j = (0,\ldots,0,\underbrace{1}_{j\text{-th coordinate}},0,\ldots,0)$. 
Then  
\begin{equation}\label{eq:recession1}
\operatorname{cone}\{-e_j:j\in I\} \subset  \Rec(L).
\end{equation}
Moreover, if $\Omega$ is bounded, then 
\begin{equation}\label{eq:recession}
 \Rec(L)=\operatorname{cone}\{-e_j:j\in I\}.
\end{equation}
\end{lemma}

\begin{proof}
First, suppose that \(j\in I\). 
If $x=(\log|z_1|,\ldots,\log|z_n|)\in L,$ 
then
$x-te_j = (\log|z_1|,\ldots,\log|z_j|-t,\ldots,\log|z_n|) \in L $
for every \(t\geq 0\) by the relative completeness. Hence
$-e_j\in\operatorname{recc}(L)$ and thus 
$\operatorname{cone}\{-e_j:j\in I\}
\subseteq
\operatorname{recc}(L)$ by the definition of the recession cone.

When $\Omega$ is bounded, we also prove the reverse inclusion. 
We first claim that if \(j\notin I\),
then
\begin{equation}\label{inf}
\inf_{z\in\Omega}|z_j|>0.
\end{equation}
Suppose, to the contrary, that this infimum were zero. Then there would
exist a sequence \(z^{(k)}\in\Omega\) such that
$z_j^{(k)} \rightarrow 0.$
Since \(\Omega\) is bounded, after passing to a subsequence we may assume
that $z^{(k)}\rightarrow p=(p_1,\ldots,p_n)\in\overline{\Omega}$.
Then \(p_j=0\). Since \(j\notin I\), 
\(p\in\partial\Omega\).
Let \(\rho\) be a smooth, torus-invariant local defining function for
\(\Omega\) near \(p\), so that
$\Omega=\{\rho<0\}$ and $d\rho(p)\neq 0.$
Invariance under rotations in the \(z_j\)-variable implies that
$\frac{\partial\rho}{\partial z_j}(p)=0.$
Let $H_j=\{z\in\mathbb{C}^n:z_j=0\}.$
Then the restriction of \(d\rho(p)\)
to \(T_pH_j\) is nonzero. Hence there exists a real vector
\(v\in T_pH_j\) such that $d\rho(p)[v]<0.$
Because \(H_j\) is a linear subspace, \(p+tv\in H_j\) for every real
number \(t\). Moreover,  
by Taylor expansion, 
\[
\rho(p+tv)
=
\rho(p)+t\,d\rho(p)[v]+o(t)
=
t\,d\rho(p)[v]+o(t).
\]
Therefore, $\rho(p+tv)<0$ and thus \(p+tv\in\Omega\) for all sufficiently small \(t>0\).
Therefore, $ p+tv\in\Omega\cap\{z_j=0\},$
which contradicts \(j\notin I\). This proves (\ref{inf}).

Now let $d=(d_1,\ldots,d_n)\in\operatorname{recc}(L).$
For every \(x\in L\) and every \(t\geq0\), we have $x+td\in L.$
Since \(\Omega\) is bounded, there exists \(R>0\) such that
$ x_k<\log R.$
If \(d_k>0\), then
$ x_k+td_k\rightarrow+\infty$
as $t\rightarrow\infty,$
which is impossible. Therefore
$ d_k\leq0$ for every $k$.
If \(j\notin I\),  it follows from (\ref{inf}) that there exists a constant \(\delta_j>0\) such that
$ |z_j|\geq\delta_j$
for every $z\in\Omega$ and thus $x_j\geq\log\delta_j$ for every $x \in L$. 
If \(d_j<0\), then $ x_j+td_j\rightarrow-\infty$
as $t\rightarrow\infty,$ which is impossible.
Therefore,
$d = \sum_{j\in I}(-d_j)(-e_j) \in\operatorname{cone}\{-e_j:j\in I\}.$
The lemma is proved combining the two inclusions.
\end{proof}

The following results due to Hsiao-Huang-Li (cf. Theorem 2.9 in \cite{HHL26}) are of fundamental importance in our proof.

\begin{theorem}[Hsiao-Huang-Li]\label{hhl1}
Let \(\Omega \subset \mathbb{C}^{n}\) be a possibly unbounded
pseudoconvex domain and let \(p \in \partial \Omega\) be a smooth,
strongly pseudoconvex boundary point. Assume that the Bergman metric
of \(\Omega\) is Einstein. Then \(\partial \Omega\) is
spherical near \(p\).
\end{theorem}

Moreover, 
Lemma 3.2 in \cite{HHL26}, applied at a smooth strongly pseudoconvex boundary point $p$ of $\Omega$, yields
\begin{equation}\label{eq:Jconstant}
 J_\Omega \equiv \frac{(n+1)^n\pi^n}{n!}.
\end{equation}
In particular, the Bergman-Einstein metric $\omega_\Omega$ on $\Omega$ satisfies 
$ \Ric(\omega_\Omega)=-\omega_\Omega$ using  \eqref{eq:Jconstant}.

\begin{theorem}[Hsiao-Huang-Li]\label{hhl2}
Let \(\Omega\subset\C^n$ be a bounded domain with
real analytic boundary.  If the Bergman metric of $\Omega$ is a complete Einstein metric, then $\Omega$ is biholomorphic
to the complex unit ball $\B^n$.
\end{theorem}

\section{Logarithmic geometry of a Reinhardt domain}
Assume, from now on, that $n\geq2$ and that $\Omega$ satisfies
the hypotheses of Theorem~\ref{thm:main}.  
Since $\Omega$ admits a complete Bergman metric, $\Omega$ is pseudoconvex \cite{B55}. 
We have the following standard result in the literature.

\begin{proposition}\label{prop:connected}
The hypersurfaces $S$ and $M$ are connected.
\end{proposition}

\begin{proof}
It follows from the argument above that both $S$ and $M$ 
are smooth manifolds.
Since the Bergman metric $\omega_\Omega$ is well defined on $\Omega$, $A^2(\Omega)\not=\{0\}$. By Lemma \ref{lem:no-line}, 
$L_\Omega$ does not contains an affine line. 
It is well known that the boundary of a line-free open convex
subset of $\mathbb{R}^n$, $n\geq 2$, is connected (cf. Lemma 2.2 and the preceding paragraph in \cite{G12}). In fact, it is
homeomorphic to either $S^{n-1}$ or $\mathbb{R}^{n-1}$ depending on
whether the convex set is bounded or unbounded, respectively.
 Since $L$ is convex,  \(S\) and
thus \(M\), are connected.
\end{proof}

We record the following standard consequence of convex geometry and
the logarithmic geometry of Reinhardt domains (see, for example,
\S 8 and \S 14 in \cite{R70} and Lemma 4.5 in \cite{G12}).

\begin{lemma}\label{lem:strong-point}
The boundary $\partial\Omega$ has a strongly pseudoconvex point whose
coordinates are all nonzero.
\end{lemma}

\begin{proof}
By Lemma~\ref{lem:no-line}, $L$ does not have an affine line. It is a standard fact in convex geometry that $C:=\Rec(L)$ is pointed, i.e. $C\cap(-C)=\{0\}.$ (cf. \cite{R70}).
Then the polar cone
\[
C^\circ=\{\nu\in\mathbb R^n:\nu\cdot q\leq0
\text{ for all }q\in C\}
\]has nonempty interior.
Therefore, $ U:=\left(\operatorname{int}C^\circ \right)\cap S^{n-1}$
is a nonempty open subset of \(S^{n-1}\).  Since, for every \(\nu\in U, q\in C\setminus\{0\}\),
$ \nu\cdot q<0$, it
 follows that \(y\mapsto\nu\cdot y\) attains its maximum on
\(\overline L\). 
The maximum is attained at some \(y_\nu\in S\). It follows from
the smoothness of \(S\) that its outward unit normal there is
\(\nu\).
Thus the image of the Gauss map $ N:S\rightarrow S^{n-1}$
contains \(U\).  By Sard's theorem,
 \(dN\) is nonsingular at some \(y_0\in S\).  
 Since $L$ is convex, 
\(dN\) must be positive definite at
\(y_0\). Namely, the hessian of the defining function of $S$ is positive definite restricted to $T_{y_0} S$. 
Since
\[
 \operatorname{Exp}(w)=(e^{w_1},\ldots,e^{w_n})
\]
is locally biholomorphic, \(\partial\Omega\) is strongly pseudoconvex at
\(\operatorname{Exp}(y_0)\) by straightforward calculation.
In particular, all coordinates of \(\operatorname{Exp}(y_0)\) are  nonzero.
\end{proof}

\begin{proposition}\label{prop:sphere}
The tube $M=S+i\R^n$ is a connected closed strongly pseudoconvex
spherical tube hypersurface.
\end{proposition}

\begin{proof}
Let
\[
 U=\{x\in S:M\text{ is strongly pseudoconvex at }x+i0\}.
\]
By Lemma \ref{lem:strong-point} and the local biholomorphism of $\operatorname{Exp}$, $U$ is a nonempty 
 open subset of $S$.  
By  the localization argument, 
$\partial\Omega$ is locally spherical near $\operatorname{Exp}(x)$ (cf. Theorem 2.9 in \cite{HHL26}).
Using the local biholomorphism of $\operatorname{Exp}$ and invariance under imaginary
translations, $U+i\R^n$ is locally spherical.

Let $C$ be a connected component of $U$.  The connected tube
$C+i\R^n$ is locally spherical.
It follows from \cite{I93} (cf. also Theorems~4.1--4.2 in \cite{I11}) that
$C+i\R^n$ can be embedded in a closed nonsingular real-analytic tube
hypersurface. Let $M^{\mathrm{ext}}$ be the connected component of this
extension containing $C+i\R^n$ and $M^{\mathrm{ext}}$ is locally CR-equivalent to a
nondegenerate hyperquadric.
In particular, the signature of the Levi form is constant. 
Since the Levi form is positive definite along $C+i\R^n$, it is positive
definite everywhere on $M^{\mathrm{ext}}$.

We claim that $C$ is closed in $S$.  Let $q\in S\cap\overline C$ and
choose $q_\nu\in C$ with $q_\nu\to q$.  Since
$q_\nu+i0\in M^{\mathrm{ext}}$ and $M^{\mathrm{ext}}$ is closed, we
have $q+i0\in M^{\mathrm{ext}}$. Along $C+i\R^n$, the hypersurfaces
$M$ and $M^{\mathrm{ext}}$ coincide. Choose compatible local
coorientations. Their tangent spaces and Levi forms agree at each
$q_\nu+i0$; smoothness of both hypersurfaces and passage to the limit show
that they also agree at $q+i0$. The Levi form of $M$ at $q+i0$ is therefore
positive definite. Hence $q\in U$. A sufficiently small connected
neighborhood of $q$ in $U$ meets $C$ and 
thus $q\in C$, proving that $C$
is closed in $S$.

As a component of the open set $U$ in the smooth manifold $S$,  $C$ is open in $S$.  Since $S$ is connected, it follows that
$C=S$.  Consequently
\[
 M=S+i\R^n=C+i\R^n\subset M^{\mathrm{ext}}.
\]
Both $M$ and $M^{\mathrm{ext}}$ are smooth hypersurfaces of the same dimension, so $M$ is open
in $M^{\mathrm{ext}}$.  Because $S$ is closed
in $\R^n$, $M$ is also closed in $M^{\mathrm{ext}}$.  The connectedness of $M^{\mathrm{ext}}$ now yields
$M=M^{\mathrm{ext}}$.  Therefore $M$ is locally spherical everywhere.
\end{proof}

\section{A sharp moment inequality and the characterization of Reinhardt domains}
In this section, we prove the following characterization of Reinhardt domains,
which may be of independent interest.

\begin{theorem}\label{thm:origin-case}
Assume $n\geq2$.  Let $\Omega\subset\C^n$ be a possibly unbounded smooth
Reinhardt domain with a well-defined complete Bergman-Einstein metric.
If $0\in\Omega$, then $\Omega$ is an ellipsoid and hence is
biholomorphic to $\B^n$.
\end{theorem}

\begin{lemma}[Sharp moment inequality]\label{lem:moment}
Let $E\subset\R_+^n$ be measurable and suppose
\[
 0<V=|E|<\infty,
 \quad
 m_j=\int_E x_j\,dx\in(0,\infty).
\]
Then
\begin{equation}\label{eq:moment-ineq}
 \frac{V^{n+1}}{\prod_{j=1}^n m_j}
 \leq\frac{(n+1)^n}{n!}.
\end{equation}
Equality holds if and only if, modulo a set of Lebesgue measure zero,
\[
 E=\left\{x\in\R_+^n:\sum_{j=1}^n a_jx_j<1\right\}
\]
for some $a_1,\ldots,a_n>0$.
\end{lemma}

\begin{proof}
Set
\[
a_j=\frac{V}{(n+1)m_j},
\quad
\ell(x)=\sum_{j=1}^{n}a_jx_j.
\]
Then
\[
\int_E \ell(x)\,dx
 =\sum_{j=1}^{n}a_jm_j
 =\frac{nV}{n+1}.
\]
Choose \(r>0\) such that $S_r:=\left\{x\in\mathbb{R}_+^n:\ell(x)<r\right\}$ 
has volume \(V\). Since \(|E|=|S_r|\), we have $|E\setminus S_r|=|S_r\setminus E|.$
Consequently,
\begin{align*}
\int_E\ell(x)\,dx-\int_{S_r}\ell(x)\,dx
&=
\int_{E\setminus S_r}\ell(x)\,dx
-\int_{S_r\setminus E}\ell(x)\,dx\\
&=
\int_{E\setminus S_r}\bigl(\ell(x)-r\bigr)\,dx
+\int_{S_r\setminus E}\bigl(r-\ell(x)\bigr)\,dx \\
&\geq 0,
\end{align*}
since $\ell(x)\geq r $ on $E\setminus S_r$ and 
$ \ell(x)<r$ on $S_r\setminus E.$
A change of variables given by \(y_j=\frac{a_jx_j}{r}\) yields
$ |S_r| =\frac{r^n}{n!\prod_{j=1}^{n}a_j}$
and
$ \int_{S_r}\ell(x)\,dx
=\frac{nr}{n+1}|S_r|
=\frac{nrV}{n+1}.$
It follows that
\[
\frac{nV}{n+1}
=\int_E\ell(x)\,dx
\geq
\int_{S_r}\ell(x)\,dx
=\frac{nrV}{n+1},
\]
and therefore $r\leq 1.$
Moreover, 
\[
V=|S_r|
=\frac{r^n}{n!\prod_{j=1}^{n}a_j}\leq\frac{1}{n!\prod_{j=1}^{n}a_j}\quad{\rm and}\quad
\prod_{j=1}^{n}a_j
=
\frac{V^n}{(n+1)^n\prod_{j=1}^{n}m_j}.
\]
imply
\[
\frac{V^{n+1}}{\prod_{j=1}^{n}m_j}
\leq
\frac{(n+1)^n}{n!}.
\]

Suppose equality holds. Then \(r=1\), and thus 
\[
\int_E\ell(x)\,dx= \int_{S_1}\ell(x)\,dx.
\]
Hence
\[
\int_{E\setminus S_1}\bigl(\ell(x)-1\bigr)\,dx
+
\int_{S_1\setminus E}\bigl(1-\ell(x)\bigr)\,dx
=0.
\]
Therefore $S_1\setminus E$ has Lebesgue measure zero, while
\(E\setminus S_1\) is contained, modulo a set of Lebesgue measure zero, in
the hyperplane $\{x\in\mathbb{R}_+^n:\ell(x)=1\}$, which also has
Lebesgue measure zero.
Therefore 
\[
E=S_1
=
\left\{
x\in\mathbb{R}_+^n:
\sum_{j=1}^{n}a_jx_j<1
\right\}
\]
modulo a set of Lebesgue measure zero. The converse follows by direct
integration.
\end{proof}

\begin{proof}[Proof of Theorem \ref{thm:origin-case}]
Let 
\[
 V_\Omega=\Vol(\Omega),
 \quad
 M_j=\int_\Omega|z_j|^2\,dV(z).
\]
We first prove that \(V_\Omega\) and \(M_j\) are finite.  Since
\(\Omega\) is Reinhardt, the torus action
\[ U_\theta f(z) := f(e^{i\theta_1}z_1,\ldots,e^{i\theta_n}z_n),
\quad
\theta\in[0,2\pi]^n,
\]
is unitary on \(A^2(\Omega)\).  For
\(\alpha\in\mathbb Z_{\geq0}^n\), define the corresponding torus
projection by
\[ P_\alpha f(z) := \frac{1}{(2\pi)^n} \int_{[0,2\pi]^n} e^{-i\alpha\cdot\theta}U_\theta f(z)\,d\theta. \]
If $f(z)=\sum_{\alpha\in\mathbb Z_{\geq0}^n}a_\alpha z^\alpha$
is the Taylor expansion of \(f\) at the origin, then
$P_\alpha f(z)=a_\alpha z^\alpha.$
Moreover, \(P_\alpha\) is an orthogonal projection, and hence
$ \|P_\alpha f\|_{L^2(\Omega)} \leq \|f\|_{L^2(\Omega)}.$

Suppose first that \(V_\Omega=\infty\).  Taking \(\alpha=0\), we have
$P_0f=a_0=f(0).$ 
Since \(P_0f\in A^2(\Omega)\),
\[\|P_0f\|_{L^2(\Omega)}^2 = |f(0)|^2V_\Omega<\infty.\]
This forces $f(0)=0$
for every $f\in A^2(\Omega).$
Therefore \(K_\Omega(0,0)=0\).  This contradicts the assumption
that the Bergman metric is well-defined at the origin.  
Hence $V_\Omega<\infty.$

Now fix \(j\in\{1,\ldots,n\}\) and suppose that \(M_j=\infty\).  For
the multi-index \(e_j\), the torus projection is
$ P_{e_j}f(z) = a_{e_j}z_j = \frac{\partial f}{\partial z_j}(0)\,z_j.$
Consequently,
\begin{align*}
\|P_{e_j}f\|_{L^2(\Omega)}^2 = \left| \frac{\partial f}{\partial z_j}(0) \right|^2 \int_\Omega|z_j|^2\,dV(z)= \left| \frac{\partial f}{\partial z_j}(0) \right|^2M_j.
\end{align*}
Since \(P_{e_j}f\in A^2(\Omega)\), 
 \(M_j=\infty\) implies
$\frac{\partial f}{\partial z_j}(0)=0$ 
for every $f\in A^2(\Omega).$ It then follows from the extremal property that  the Bergman metric is degenerate in the \(z_j\)-direction,
which contradicts to the hypothesis.  Hence
\(M_j<\infty\).  Since \(j\) was arbitrary, we conclude that
 \(M_j\) is finite for every  $1\leq j\leq n$.

Define
\[
 D=\{(|z_1|^2,\ldots,|z_n|^2):z\in\Omega\}\subset\R_+^n.
\]
A straightforward calculation using polar coordinates yields
\begin{equation}\label{eq:polar-moments}
 V_\Omega=\pi^n|D|,
 \quad
 M_j=\pi^n\int_Dx_j\,dx.
\end{equation}
It is a standard fact that holomorphic monomials are mutually orthogonal in
the Bergman space of a Reinhardt domain. Since $0\in\Omega$, a direct
calculation from the definition using the constant and linear terms of the
diagonal Bergman kernel yields
\[
 K_\Omega(0)=\frac1{V_\Omega},
 \quad
 g_{j\bar k}(0)=\delta_{jk}\frac{V_\Omega}{M_j}.
\]
It follows that the Bergman invariant function
$J_\Omega=\frac{\det(g_{j\bar k})}{K_\Omega}$ satisfies
\begin{equation}\label{eq:J-origin}
 J_\Omega(0)
 =\frac{V_\Omega^{n+1}}{\prod_jM_j}
 =\pi^n\frac{|D|^{n+1}}
                  {\prod_j\int_Dx_j\,dx}.
\end{equation}
Lemma~\ref{lem:moment} therefore yields
\[
 J_\Omega(0)\leq\frac{(n+1)^n\pi^n}{n!}.
\]
By Lemma 3.2. in \cite{HHL26} (cf. \cite{FW97, SX25, Y25}), 
$J_\Omega\equiv \frac{(n+1)^n\pi^n}{n!}$. 
Thus equality holds
in Lemma~\ref{lem:moment}, and $D$ agrees almost everywhere with a
simplex
\[ \Sigma:= \left\{x\geq0:\sum_{j=1}^na_jx_j<1\right\}.\]

We first prove that $D$ and $\Sigma$ agree exactly on the strictly
positive orthant $P:=(0,\infty)^n.$
Define
\[
U:=\log(D\cap P),
\quad
V:=\log(\Sigma\cap P),
\]
where $\log(x_1,\ldots,x_n) := (\log x_1,\ldots,\log x_n).$
Since \(x_j=|z_j|^2\), we have
\[
U
=
\left\{
2(\log|z_1|,\ldots,\log|z_n|):
z\in\Omega\cap(\mathbb{C}^{*})^n
\right\}.
\]
Thus \(U\) is open and convex by the logarithmic convexity of the
pseudoconvex Reinhardt domain \(\Omega\). On the other hand,
\[
V
=
\left\{
y\in\mathbb{R}^{n}:
\sum_{j=1}^{n}a_je^{y_j}<1
\right\}.
\]
The function $y\mapsto\sum_{j=1}^{n}a_je^{y_j}$ is continuous and convex,
so $V$ is also open and convex.
Since \(D\) and \(\Sigma\) agree almost everywhere and the logarithm
is locally Lipschitz on \(P\), it follows that \(U\) and \(V\) agree
almost everywhere. They therefore have the same closure: if, for example,
$u\in U\setminus\overline V$, then a small ball about $u$ lies in
$U\setminus V$, contradicting almost-everywhere equality; the reverse
inclusion is identical. Moreover, every nonempty open convex subset of
\(\mathbb{R}^n\) is the interior of its closure (cf. Theorem 6.3 in \cite{R70}).
It follows that \(U=V\) and thus $D\cap P=\Sigma\cap P$.

It remains to recover the coordinate faces. 
Since \(\Omega\) is a pseudoconvex Reinhardt domain containing the
origin, it is a complete Reinhardt domain (cf. Theorem~1.11.13 in
\cite{JP08}). Thus \(D\) is
coordinatewise complete. The following is the standard mechanism for recovering
the coordinate faces (comparing the proof of
\((iii)\Rightarrow(iv)\) in
Theorem 1.11.13, pp. 80--81, \cite{JP08}).
Let \(x\in\Sigma\). If some coordinates of \(x\) vanish, replace them
by sufficiently small positive numbers. Since the defining inequality
of \(\Sigma\) is strict, this produces a point
\(x^{\varepsilon}\in\Sigma\cap P=D\cap P\). Coordinate completeness
then allows us to set those coordinates back equal to zero, and hence
\(x\in D\). Therefore, $\Sigma\subset D.$
Conversely, let \(x\in D\), and choose \(z\in\Omega\) with
\(x_j=|z_j|^2\). By openness, the zero coordinates of \(z\) may be
replaced by sufficiently small nonzero values while remaining in
\(\Omega\). We thus obtain \(x^\nu\in D\cap P=\Sigma\cap P\) with
\(x^\nu\to x\), and consequently $\sum_{j=1}^{n}a_jx_j\leq1.$
We now show equality is impossible. If equality held, then \(x\neq0\);
choosing \(k\) with \(x_k>0\), we could slightly increase
\(|z_k|\), and simultaneously perturb the zero coordinates to nonzero
values, while remaining in the open set \(\Omega\). The resulting
point \(y\in D\cap P\) would satisfy $\sum_{j=1}^{n}a_jy_j>1,$
contradicting \(D\cap P=\Sigma\cap P\). Hence
$\sum_{j=1}^{n}a_jx_j<1,$
so \(x\in\Sigma\). Therefore \(D=\Sigma\).

Finally, the torus invariance implies that  \(\Omega\)
is determined entirely by the coordinate moduli. 
Hence
\[
 \Omega=\left\{z\in\C^n:\sum_{j=1}^na_j|z_j|^2<1\right\},
\]
which is biholomorphic to $\B^n$.
\end{proof}

\section{On bounded Reinhardt domains}\label{sec:bounded}

In this section, we prove the following theorem.

\begin{theorem}\label{bounded}
Let $\Omega\subset\C^n, n\geq 2, $ be a bounded Reinhardt domain with
$C^\infty$-smooth boundary.  If the Bergman metric of $\Omega$ is a 
complete Einstein metric, then $\Omega$ is biholomorphic to the complex unit ball
$\B^n$.
\end{theorem}

\begin{proof}
We re-index the original coordinates as
$z_0,\dots,z_{n-1}$, with $y_j=\log|z_j|$. 
By Proposition \ref{prop:connected} and Proposition \ref{prop:sphere}, $M$ is a connected closed strongly pseudoconvex
spherical tube hypersurface.
 After an invertible real affine
transformation
\begin{equation}\label{eq:affine-change}
 y=Ax+b,
\end{equation}
the hypersurface $A^{-1}(S-b)$ is one of three  models in Theorem \ref{DY} by \cite{DY85}.
Let
$\widetilde L=A^{-1}(L-b)$
and $H=\partial\widetilde L=A^{-1}(S-b).$
Let \(Q\) be the connected component of
\(\mathbb R^n\setminus H\) containing \(\widetilde L\). If $\widetilde L\neq Q$, then its relative boundary in the connected
set $Q$ would be nonempty, producing a point of
$\partial\widetilde L\setminus H$, a contradiction.  Thus
$\widetilde L=Q$ and hence $L=AQ+b.$ 
By Remark \ref{DY2}, the convexity of
\(\widetilde L=Q\) determines the relevant side of each
Dadok--Yang model.  Thus \(Q\) is one of $L_{\mathrm I}^{(m)}, L_{\mathrm{II}}, L_{\mathrm{III}}.$
We analyze these three possibilities separately.

\medskip

The Type I model is
\begin{equation*}
L^{(m)}_{\rm I} = \left\{x\in\mathbb{R}^n: x_0>F(x'):=\sum_{\alpha=1}^{m}e^{x_\alpha}
              +\sum_{\alpha=m+1}^{n-1}x_\alpha^2 \right\},
 \quad 0\leq m\leq n-1.
\end{equation*}
Every original logarithmic coordinate $y_j$ is bounded above on $L$.
Write
\[
 y_j=a_{j0}x_0+a'_j\cdot x'+b_j,
\]
where $x'=(x_1,\ldots,x_{n-1})$.
Letting $x_0\to+\infty$ and keeping $x'$ fixed yields $a_{j0}\leq0$.  If
$a_{j0}=0$, the fact that $x'$ ranges over all of $\R^{n-1}$ would force
$a'_j=0$, making a row of $A$ zero.  Therefore
$ a_{j0}<0$ for every $j$.
On the boundary $x_0=F(x')$, choose a path satisfying
$ F(x')\rightarrow\infty,  |x'|=o(F(x')).$
If $m>0$, one may take $x_1=t$ in an exponential direction and keep
the other coordinates fixed, yielding $F\sim e^t$.  If $m=0$, take
$x_1=t$ in a quadratic direction, yielding $F\sim t^2$.  Since $a_{j0}<0$, every $y_j$ tends to $-\infty$.  The
corresponding points of $\partial\Omega\cap(\C^*)^n$ therefore tend
to the origin.  Hence $0\in\partial\Omega$, contrary to
Lemma~\ref{lem:origin-not-boundary}. Therefore Type I is impossible.

The Type III model is
\[
L_{\mathrm{III}}
=
\left\{
x\in\mathbb{R}^n:
\sum_{\alpha=0}^{n-1}e^{x_\alpha}<1
\right\}.
\]
If \(d_\alpha\leq0\) for every \(\alpha\), then, for every
\(x\in L_{\mathrm{III}}\) and \(t\geq0\),
\[
\sum_{\alpha=0}^{n-1}e^{x_\alpha+td_\alpha}
\leq
\sum_{\alpha=0}^{n-1}e^{x_\alpha}
<1.
\]
Hence \(x+td\in L_{\mathrm{III}}\), so
\(-\mathbb{R}_+^n\subseteq\operatorname{recc}(L_{\mathrm{III}})\).
Conversely, if \(d_k>0\) for some \(k\), then for any
\(x\in L_{\mathrm{III}}\),
$e^{x_k+td_k}\rightarrow+\infty$ as $t\to+\infty$.
Consequently, \(x+td\notin L_{\mathrm{III}}\) for all sufficiently
large \(t\). Thus \(d\notin\operatorname{recc}(L_{\mathrm{III}})\).
This shows that $\operatorname{recc}(L_{\mathrm{III}})=-\R_+^n$.  Since every $y_j$ in
\eqref{eq:affine-change} is bounded above, testing the recession direction 
$-e_\alpha$ yields $A_{j\alpha}\geq0$ for all $j,\alpha.$
Every row of $A$ is nonzero, so every row sum is positive.  Along
$x=(-t,\dots,-t) \in L_{\mathrm{III}}$ for large $t$, all
$y_j\to-\infty$.  Thus $0\in\overline\Omega$.  By
Lemma~\ref{lem:origin-not-boundary}, $0\notin\partial\Omega$, so
$0\in\Omega$. It thus follows from Theorem \ref{thm:origin-case} that $\Omega$
is biholomorphic to $\B^n$.

Now consider the type II  model 
\begin{equation}\label{eq:type-II}
 L_{\rm II} = \left\{x\in\mathbb{R}^n:
 0<x_0<\pi,
 \quad
 \sum_{\alpha=1}^{n-1}e^{x_\alpha}<\sin x_0\right\}.
\end{equation}
We first show that $\operatorname{recc}(L_{\mathrm{II}}) = \{0\}\times\mathbb{R}_{-}^{n-1}.$
Let $d=(0,d_1,\ldots,d_{n-1}), d_\alpha\leq0.$
For every \(x=(x_0,x')\in L_{\mathrm{II}}\) and \(t\geq0\), the
\(x_0\)-coordinate is unchanged, while
$e^{x_\alpha+td_\alpha}\leq e^{x_\alpha}$
for each $\alpha=1,\ldots, n-1.$
Consequently,
\(x+td\in L_{\mathrm{II}}\) for every \(t\geq0\), proving
$ \{0\}\times\mathbb{R}_{-}^{{n-1}} \subseteq \operatorname{recc}(L_{\mathrm{II}}).$
Conversely, let
$ d=(d_0,d_1,\ldots,d_{n-1}) \in\operatorname{recc}(L_{\mathrm{II}}).$
For any \(x\in L_{\mathrm{II}}\), we must have
$0<x_0+td_0<\pi $ for every $t\geq0$. This forces $d_0=0.$
Suppose that \(d_k>0\) for some \(k\in\{1,\ldots, n-1\}\). Then
$e^{x_k+td_k}\rightarrow+\infty$ as $t\to+\infty,$
whereas \(\sin x_0\) remains fixed. This is impossible. It follows that
$ d\in\{0\}\times\mathbb{R}_{-}^{n-1}$ and thus
$\operatorname{recc}(L_{\mathrm{II}}) =\{0\}\times\mathbb{R}_{-}^{n-1}.$

By Lemma \ref{lem:recession} and $\operatorname{recc}(L) = A\bigl(\operatorname{recc}(L_{\mathrm{II}})\bigr)$,
the cone \(A(\operatorname{recc}(L_{\mathrm{II}}))\) is a coordinate
cone. Since
$\dim \operatorname{recc}(L_{\mathrm{II}})=n-1,$
and \(A\) is invertible, Lemma \ref{lem:recession} implies that
$\#I=n-1$.
Thus \(\Omega\) meets exactly \(n-1\) coordinate hyperplanes.
After permuting the \(y\)-coordinates, we may assume that
\(I=\{1,\ldots,n-1\}\). Thus \(A\) yields a linear isomorphism
\[
A\colon
\operatorname{cone}\{-e_1,\ldots,-e_{n-1}\}
\rightarrow
\operatorname{cone}\{-e_1,\ldots,-e_{n-1}\}.
\]
An invertible linear map taking one convex cone onto another sends
extreme rays to extreme rays. Since the extreme rays of both cones are the rays
$\mathbb{R}_{+}(-e_1),\ldots,\mathbb{R}_{+}(-e_{n-1})$ and  the Type II model is symmetric in \(x_1,\ldots,x_{n-1}\), we may
relabel these coordinates such that the affine transformation
\(y=Ax+b\) has the form
\begin{equation}
\label{eq:type-II-affine}
\begin{aligned}
y_0&=a_0x_0+b_0,
    &&a_0\neq0,\\
y_\alpha&=a_\alpha x_0+\lambda_\alpha x_\alpha+b_\alpha,
    &&\alpha=1,\ldots, n-1,
    \quad \lambda_\alpha>0.
\end{aligned}
\end{equation}
Let
\[
 t=\frac{\log|z_0|-b_0}{a_0}=x_0.
\]
Since $0<t<\pi$, the coordinate $z_0$ lies in an annulus whose closure
is compact and disjoint from zero.  On
$\Omega\cap(\C^*)^n$, substitution of
\eqref{eq:type-II-affine} into \eqref{eq:type-II} yields
\begin{equation}\label{eq:type-II-domain}
 0<t<\pi,
 \quad
 \sum_{\alpha=1}^{n-1}
 C_\alpha e^{-\frac{a_\alpha}{\lambda_\alpha} t}
 |z_\alpha|^{q_\alpha}<\sin t,
\end{equation}
where $C_\alpha=e^{-\frac{b_\alpha}{\lambda_\alpha}}>0, q_\alpha=\frac1{\lambda_\alpha}>0$.

Define $t(z):=\frac{\log |z_0|-b_0}{a_0}, \Phi(z):= \sum_{\alpha=1}^{n-1} C_\alpha e^{-\frac{a_\alpha}{\lambda_\alpha}t(z)} |z_\alpha|^{q_\alpha}.$
Let
\[
\mathcal D:=
\left\{
z\in\mathbb C^n:
z_0\neq 0,\quad
0<t(z)<\pi,\quad
\Phi(z)<\sin t(z)
\right\},
\]
where zero values among \(z_1,\ldots,z_{n-1}\) are allowed. 
It follows from the argument above that  $ \Omega\cap(\mathbb C^*)^n=\mathcal D\cap(\mathbb C^*)^n.$
We first prove that \(\mathcal D\subset\Omega\). Let \(z\in\mathcal D\)
and let $J=\{\alpha\in\{1,\ldots, n-1\}:z_\alpha=0\}.$
For every sufficiently small \(\varepsilon>0\), replace each zero
coordinate \(z_\alpha\), \(\alpha\in J\), by a nonzero number
\(z_\alpha^{(\varepsilon)}\) satisfying $0<|z_\alpha^{(\varepsilon)}|<\varepsilon,$
and leave the remaining coordinates unchanged. Denote the resulting
point by \(z^{(\varepsilon)}\). 
By continuity, $z^{(\varepsilon)} \in \mathcal D\cap(\mathbb C^*)^n = \Omega\cap(\mathbb C^*)^n$ 
for all sufficiently small \(\varepsilon>0\). 
Using the relative
completeness of the pseudoconvex Reinhardt domain $\Omega$ in the coordinate
directions belonging to \(I=\{1,\ldots,n-1\}\), we obtain \(z\in\Omega\) and thus $\mathcal D\subset\Omega.$
On the other hand, 
let \(z\in\Omega\). Since
\(\Omega\) does not meet the coordinate hyperplane \(\{z_0=0\}\), we
have \(z_0\neq0\), so \(t(z)\) is well defined. If some of
\(z_1,\ldots,z_{n-1}\) vanish, 
by perturbing $z$, 
there exists a sequence $ z^{(\nu)} \in \Omega\cap(\mathbb C^*)^n = \mathcal D\cap(\mathbb C^*)^n$
such that $ z^{(\nu)}\rightarrow z.$
Passing to the limit, we obtain
$ 0\leq t(z)\leq\pi, \Phi(z)\leq\sin t(z).$
If \(t(z)=0\) or \(t(z)=\pi\), we may perturb $z$ to get $\tilde z \in \Omega\cap(\mathbb C^*)^n$ while $t(\tilde z) <0$ or   $t(\tilde z) >\pi$, which is impossible. This implies that $0 < t(z) < \pi$. Moreover, suppose, to the contrary, that $\Phi(z)=\sin t(z).$
If \(z_\beta\neq0\) for some \(\beta\in\{1, \ldots, n-1\}\), replace
\(z_\beta\) by $z_\beta^{(\varepsilon)}:=
(1+\varepsilon)z_\beta$  and leave all other coordinates unchanged. 
For sufficiently small
\(\varepsilon>0\),  \(z^{(\varepsilon)} \in \Omega \)  by openness. Moreover, \(t(z^{(\varepsilon)})=t(z)\),
whereas \(q_\beta>0\) implies
\[ \Phi(z^{(\varepsilon)}) = \Phi(z) + C_\beta e^{-\frac{a_\beta}{\lambda_\beta} t(z)} \left((1+\varepsilon)^{q_\beta}-1\right) |z_\beta|^{q_\beta} > \Phi(z) = \sin t(z) = \sin t(z^{(\varepsilon)}). \]
This contradicts the inequality
$\Phi(w)\leq\sin t(w)$ established above for every $w\in\Omega$.
In addition, if $z_1=\cdots=z_{n-1}=0,$
then $\Phi(z)=0<\sin t(z)$ since \(0<t(z)<\pi\).
Therefore,  in every case \(z\in\mathcal D\) by the definition.
To conclude, we obtain that \(\mathcal D = \Omega\).
By the argument above, we point out that 
the endpoint orbits $ E_0=\{|z_0|=e^{b_0},\ z'=0\},  E_\pi=\{|z_0|=e^{a_0\pi+b_0},\ z'=0\}$, corresponding to $t(z)=0$ or $t(z)=\pi$, 
belong to $\partial\Omega$.

Fix a point $p\in E_0$.
Near \(p\), we use the smooth coordinates $(\theta_0,t,z'), z_0=e^{a_0t+b_0}e^{i\theta_0}, z'=(z_1,\ldots,z_{n-1}).$
Let \(\delta\) be a smooth torus-invariant defining function for
\(\Omega\) near \(p\), so that $\Omega=\{\delta<0\},  \partial\Omega=\{\delta=0\},  d\delta(p)\neq0.$
For every \(\alpha\in\{1,\ldots, n-1\}\), rotations in the
\(z_\alpha\)-plane fix \(p\), since \(z_\alpha(p)=0\). The invariance
of \(\delta\) under these rotations implies that the restriction of
\(d\delta(p)\) to the \(z_\alpha\)-plane is invariant under all planar
rotations. Since the only rotation-invariant real linear functional
on this plane is zero, all the components of \(d\delta(p)\) in the
\(z_\alpha\)-directions vanish. The \(S^1\)-invariance of \(\delta\) under rotations in the
\(z_0\)-variable implies that \(\delta\) is independent of
\(\theta_0\), yielding $\frac{\partial\delta}{\partial\theta_0}(p)=0.$
It follows that $\frac{\partial\delta}{\partial t}(p)\neq0.$
It follows from the implicit function theorem that
	 the boundary $\partial\Omega$
near \(p\) is defined by  $ t=h(\theta_0,z'),$
where \(h\in C^\infty\) and $h(\theta_0,0)=0.$
It follows from the torus
 invariance that \(h\) is independent of \(\theta_0\), so we
write simply $t=h(z').$
Fix \(\alpha\in\{1,\ldots, n-1\}\) and set $z_\gamma=0$ for every $\gamma\neq\alpha.$
Using the boundary equation, we obtain
\[
\sin h(z_\alpha)
=
C_\alpha
e^{-\frac{a_\alpha}{\lambda_\alpha} h(z_\alpha)}
|z_\alpha|^{q_\alpha}.
\]
Equivalently,
\[
|z_\alpha|^{q_\alpha}
=
C_\alpha^{-1}
e^{\frac{a_\alpha}{\lambda_\alpha} h(z_\alpha)}
\sin h(z_\alpha).
\]
The function
$s\mapsto C_\alpha^{-1}e^{\frac{a_\alpha}{\lambda_\alpha}s}\sin s$ is real
analytic and has nonzero derivative at zero. Hence $|z_\alpha|^{q_\alpha}$
is smooth at $z_\alpha=0$. A radial function
$|z|^q=(x^2+y^2)^{\frac{q}{2}}$ is $C^\infty$ at the origin exactly when $q$ is a
nonnegative even integer. Since $q_\alpha>0$, it follows that
$q_\alpha=2m_\alpha$ for some $m_\alpha\in\mathbb N$ with $m_\alpha\ge1$.

Now let
\begin{equation}\label{eq:analytic-defining-function}
 \rho(z)=
 \sum_{\alpha=1}^{n-1}
 C_\alpha e^{-\frac{a_\alpha }{\lambda_\alpha}t}
 |z_\alpha|^{2m_\alpha}-\sin t,
 \quad
 t=\frac{\log|z_0|-b_0}{a_0}.
\end{equation}
Because $z_0$ stays away from zero, $t$ is real analytic in the real
and imaginary parts of $z_0$.  
It follows that $\rho$ is real
analytic near the entire boundary of $\Omega$.  
We also verify that the
 gradient of $\rho$  does not vanish on $\{\rho=0\}$. 
If $t=0$ or $\pi$, all $z_\alpha$ vanish and
$\partial_t\rho=-\cos t\neq0$.  If $0<t<\pi$ and $\rho=0$, then 
$z_\beta\neq0$ for at least one $\beta$.  In real coordinates on that
$z_\beta$-plane,
$ \nabla_{z_\beta}\rho
 =2m_\beta C_\beta e^{-\frac{a_\beta}{\lambda_\beta} t}
   |z_\beta|^{2m_\beta-2}
   (\operatorname{Re}z_\beta,\operatorname{Im}z_\beta)\neq0.$
 Moreover, on the exterior side of \(E_0\), where \(t<0\), we
have
\[
\rho(z)
=
\sum_{\alpha=1}^{n-1}
C_\alpha e^{-\frac{a_\alpha}{\lambda_\alpha} t}
|z_\alpha|^{2m_\alpha}
-\sin t
>0.
\]
Similarly, on the exterior side of \(E_\pi\), where \(t>\pi\), we
again have  \(\rho(z)>0\). Thus, near either
endpoint orbit, the condition \(\rho<0\) automatically forces
\(0<t<\pi\). Consequently, $\rho$ is a  local defining function  for \(\Omega\) near $E_0$ and $E_\pi$.
  Therefore $\Omega$ is a
bounded real-analytic pseudoconvex domain.
It then follows from Theorem 1.1 in \cite{HHL26} that
$\Omega$ is biholomorphic to $\B^n$.
This completes the proof of Theorem \ref{bounded}.
\end{proof}

\section{On unbounded Reinhardt domains}\label{sec:unbounded}

In this section, we assume that  $\Omega$ is unbounded and will prove the following theorem.

\begin{theorem}\label{unbounded}
Let $\Omega\subset\C^n, n \geq 2,$ be an unbounded Reinhardt domain with
$C^\infty$-smooth boundary.  Then the Bergman metric of $\Omega$ cannot be a
complete Einstein metric.
\end{theorem}

It follows from Theorem~\ref{thm:origin-case} that 
$ 0\notin\Omega.$
We shall eliminate all three models in Theorem \ref{DY}.

\subsection{Some preparation}

The set $ \mathcal P=\{\alpha\in\Z^n:z^\alpha\in A^2(\Omega)\}$ of $L^2$-allowable indices for the Reinhardt domain is introduced in \cite{EM26}.
Note that the square-integrable monomials form an orthogonal basis for the Bergman space $A^2(\Omega)$. By the torus symmetry, $z^\alpha, z^\beta$ are orthogonal for $\alpha, \beta \in \mathcal P, \alpha \not= \beta$. 
Denote $ c_\alpha=\|z^\alpha\|^{-2}, y_j=\Log |z_j|, y \in L = \Log \Omega$. 
The  Bergman kernel
$ \mathcal K(y):=K_\Omega(z,z)= \sum_{\alpha\in\mathcal P} c_\alpha e^{2\alpha\cdot y},$ together with all of its
derivatives, 
 converges locally normally on  \(\Omega\cap(\mathbb C^*)^n\). 
Let \(\one=(1,\ldots,1)\).  It follows from the straightforward calculation using 
polar coordinates that
\begin{equation}\label{eq:monomial-norm}
c_\alpha^{-1} =\|z^\alpha\|_{L^2(\Omega)}^2
 =
 (2\pi)^n
 \int_L e^{2(\alpha+\one)\cdot y}\,dy < \infty
\end{equation}
for every \(\alpha\in\mathcal P\).

\begin{lemma}
\begin{equation}\label{metric}
(\omega_\Omega)_{j\overline k}(z)
=
\frac{1}{4z_j\overline z_k}
\frac{\partial^2 \log \mathcal K}
{\partial y_j\partial y_k}(y)
\end{equation}
for any $z \in \Omega\cap(\mathbb C^*)^n$.
\end{lemma}

\begin{proof}
By $y_\ell=\log|z_\ell| =\frac12\log(z_\ell\overline z_\ell),$
we have
\[
\frac{\partial y_\ell}{\partial z_j}
 =
\frac{\delta_{j\ell}}{2z_j},
\quad
\frac{\partial y_\ell}{\partial\overline z_k}
 =
\frac{\delta_{k\ell}}{2\overline z_k}.
\]
Moreover, \(\log|z_\ell|\) is pluriharmonic on
\(\{z_\ell\neq0\}\), and therefore $\frac{\partial^2y_\ell}{\partial z_j\partial\overline z_k}=0.$
It follows from the chain rule that
\begin{align*}
(\omega_\Omega)_{j\overline k}(z) =
\frac{\partial^2}
{\partial z_j\partial\overline z_k}
\log K_\Omega(z,z) =
\frac{\partial^2}
{\partial z_j\partial\overline z_k} \log \mathcal K(y) =
\frac{1}{4z_j\overline z_k}
\frac{\partial^2  \log \mathcal K}
{\partial y_j\partial y_k}(y).
\end{align*}
\end{proof}

\begin{lemma}\label{lem:spectrum-spans}
For every \(y\in L\),
\[
\ker D^2 \log\mathcal K(y)
 =
\bigl(
\operatorname{span}_{\mathbb R}(\mathcal P-\mathcal P)
\bigr)^\perp.
\]
Consequently, since the Bergman metric is nondegenerate, 
\[
\operatorname{span}_{\mathbb R}(\mathcal P-\mathcal P)
=\mathbb R^n.
\]
\end{lemma}

\begin{proof}
Fix \(y\in L\).  Since $ p_y(\alpha) := \frac{c_\alpha e^{2\alpha\cdot y}}{\mathcal K(y)} >0$ for every  $\alpha\in\mathcal P$ and  $\sum_{\alpha\in\mathcal P}p_y(\alpha)=1$, 
$\{p_y(\alpha)\}$
defines a probability distribution on \(\mathcal P\). 
By term-by-term differentiation, we obtain
\[
 \frac{\partial \log\mathcal K}{\partial y_j}(y)
 =
 \frac{1}{\mathcal K(y)}
 \sum_{\alpha\in\mathcal P}
 2\alpha_jc_\alpha e^{2\alpha\cdot y}
 =
 2\sum_{\alpha\in\mathcal P}
 \alpha_jp_y(\alpha).
\]
Differentiating once more yields
\begin{align*}
 \frac{\partial^2  \log\mathcal K}{\partial y_j\partial y_k}(y)
 &=
 4\left(
 \sum_{\alpha\in\mathcal P}
 \alpha_j\alpha_kp_y(\alpha)
 -
 \left(\sum_{\alpha\in\mathcal P}
 \alpha_jp_y(\alpha)\right)
 \left(\sum_{\alpha\in\mathcal P}
 \alpha_kp_y(\alpha)\right)
 \right).
\end{align*}
It follows that, for every \(v\in\mathbb R^n\),
\begin{align*}
 v^TD^2  \log\mathcal K(y)v
 &=
 4\left(
 \sum_{\alpha\in\mathcal P}
 (\alpha\cdot v)^2p_y(\alpha)
 -
 \left(
 \sum_{\alpha\in\mathcal P}
 (\alpha\cdot v)p_y(\alpha)
 \right)^2
 \right)                                      =
 4\operatorname{Var}_{p_y}(\alpha\cdot v).
\end{align*}
Moreover, 
$\operatorname{Var}_{p_y}(\alpha\cdot v)=0$
 if and only if
\(\alpha\cdot v\) is constant on 
\(\mathcal P\).  Namely, 
$ (\alpha-\beta)\cdot v=0$
for all $\alpha,\beta\in\mathcal P.$
It follows that
\[
 \ker D^2  \log\mathcal K(y)
 =
 \left(
 \operatorname{span}_{\mathbb R}(\mathcal P-\mathcal P)
 \right)^\perp.
\]
Writing \(w=y+i\theta\), 
 it follows from the torus invariance that 
\[
 \log K_\Omega\bigl(\operatorname{Exp}(w),
                    \operatorname{Exp}(w)\bigr)
 =
  \log\mathcal K (\operatorname{Re} w)= \log\mathcal K(y).
\]
By (\ref{metric}), the pullback of the Bergman metric reads 
\[
 \frac{\partial^2}{\partial w_j\partial\overline{w}_k}  \log K_\Omega\bigl(\operatorname{Exp}(w),
                    \operatorname{Exp}(w) =  \frac{\partial^2}{\partial w_j\partial\overline{w}_k}
  \log\mathcal K(\operatorname{Re} w)
 =
 \frac14\frac{\partial^2   \log\mathcal K   }{\partial y_j\partial y_k}(y).
\]
Suppose that $ \operatorname{span}_{\mathbb R}(\mathcal P-\mathcal P) \neq\mathbb R^n.$
Then there exists \(v\in\mathbb R^n\setminus\{0\}\) orthogonal to
\(\mathcal P-\mathcal P\).  
It follows that $ D^2  \log\mathcal K(y)v=0$ for every $y\in L$.
Thus  the nonzero \((1,0)\)-vector $ V=\sum_{j=1}^n v_j\frac{\partial}{\partial w_j}$ satisfies 
$$(\operatorname{Exp}^*\omega_\Omega)(w) (V,\overline V) = \frac14v^TD^2  \log\mathcal K(y)v =0.$$
This is impossible and thus 
\[ \operatorname{span}_{\mathbb R}(\mathcal P-\mathcal P)
 =
 \mathbb R^n.
\]
\end{proof}

\begin{corollary}\label{lem:radial-formula}
Let \(y_0\in L\) and \(v\in\mathbb R^n\).  Let $I:=\{t\in\mathbb R:y_0+tv\in L\}$. 
Along the curve $ z_j(t)=z_j(0)e^{v_jt}$ with $ y(t)=y(0)+tv$, 
\begin{equation}\label{eq:radial-speed}
 |\dot z(t)|_{\omega_\Omega}^2
 =\frac14 v^TD^2 \log\mathcal K(y(t))v
 =\frac14\frac{\dd^2}{\dd t^2} \log\mathcal K(y(0)+t v).
\end{equation}
\end{corollary}

\begin{proof}
It follows from (\ref{metric}) and the straightforward calculation.
\end{proof}

The following lemma is an obstruction to Bergman completeness at
infinity.  It may be regarded as an unbounded counterpart of 
Zwonek's Theorem 3 \cite{Z99} 
 for bounded pseudoconvex Reinhardt
domains (cf. also Proposition 2.7.2 in \cite{Z00}). Related
incompleteness results for certain unbounded Hartogs domains were
proved by Pflug and Zwonek (cf. Theorem 7 in \cite{PZ05}).

\begin{lemma}
\label{lem:rational-ray}
Suppose that there is a vector $ v\in\Rec(L)\cap\Q^n$ with $ v_j>0$ for some $j$.
Then the Bergman metric $\omega_\Omega$ of $\Omega$ is incomplete.
\end{lemma}

\begin{proof}
By multiplying $v$ by a positive integer, we may assume $v \in\Z^n$.  Choose
$y_0\in L$ and a bounded open neighborhood $U$ of $0$ such that
$y_0+U\Subset L$.  Since $v \in\Rec(L)$,
$ y_0+U+tv\subset L$ for $t\geq0$.
Choose a bounded relatively open set $V\subset v^\perp$ containing $0$
and small enough that $y_0+V\subset y_0+U$.  
Then the infinite cylinder $\mathcal C:= \{y_0+u+sv:u\in V,\ s\geq0\}$
is contained in \(L\).
Let \(\alpha\in\mathcal P\) and write \(y=y_0+u+sv\).  By
\eqref{eq:monomial-norm},

\begin{align*}
\int_{\mathcal C}
 e^{2(\alpha+\one)\cdot y}\,\dd y=
 |v|e^{2(\alpha+\one)\cdot y_0}
 \left(
   \int_V e^{2(\alpha+\one)\cdot u}\,\dd\sigma(u)
 \right)
 \left(
   \int_0^\infty
   e^{2s(\alpha+\one)\cdot v}\,\dd s
 \right)<\infty,
\end{align*}
where \(\dd\sigma\) denotes the induced Lebesgue measure on
\(v^\perp\). It follows that
\[
\int_0^\infty
e^{2s(\alpha+\one)\cdot v}\,\dd s<\infty,
\]
implying 
$(\alpha+\one)\cdot v<0.$
Thus the integers $\alpha\cdot v$ are bounded above.  Let
$ M=\max_{\alpha\in\mathcal P}\alpha\cdot v$
and  
$k=M-\alpha\cdot v\in\Z_{\geq0}$. Then the Bergman kernel 
\begin{equation}\label{eq:grouped-series}
 \mathcal K(y_0+t v)
 =e^{2Mt}\sum_{k=0}^\infty A_ke^{-2kt},
\end{equation}
with 
$ A_k=\sum_{\substack{\alpha\in\mathcal P\\
 M-\alpha\cdot v=k}}c_\alpha e^{2\alpha\cdot y_0} \geq0$, $A_0>0$, and
$\sum_kA_k=\mathcal K(y_0)<\infty$.  
We claim that \(A_k>0\) for at least one \(k>0\).  Otherwise,
$\alpha\cdot v=M$
for every $\alpha\in\mathcal P$.
It  follows that $(\alpha-\beta)\cdot v=0$
for all $\alpha,\beta\in\mathcal P.$
By Lemma \ref{lem:spectrum-spans},
 \(v=0\), contradicting \(v_j>0\).
 
Let  $S(t):=\sum_{k=0}^\infty A_ke^{-2kt}$.
By \eqref{eq:grouped-series}, $\log\mathcal K(y_0+tv)=2Mt+\log S(t).$
Define probability distribution $\{p_t(k)\}$ with $p_t(k)= \frac{A_ke^{-2kt}}{S(t)}.$
It follows from term-by-term differentiation that
\[
\frac{\dd}{\dd t}\log S(t)
=
-2\sum_{k=0}^\infty k\,p_t(k)
\]
and
\[
\frac{\dd^2}{\dd t^2}\log S(t)
=
4\left(
\sum_{k=0}^\infty k^2p_t(k)
-
\left(\sum_{k=0}^\infty kp_t(k)\right)^2
\right).
\]
Consequently,
$$\frac{\dd^2}{\dd t^2}\log\mathcal K(y_0+tv) = 4\operatorname{Var}_{p_t}(k).$$
It follows from 
$\operatorname{Var}_{p_t}(k) \leq \sum_{k=0}^\infty k^2p_t(k)$
and \(S(t)\geq A_0\) that
\begin{align*}
\frac{\dd^2}{\dd t^2}\log\mathcal K(y_0+tv)
&\leq
\frac{4\sum_{k\geq1}k^2A_ke^{-2kt}}{A_0}
\end{align*}
 for \(t\geq1\).
Moreover, $k^2e^{-2kt}=e^{-t}k^2e^{-(2k-1)t} \leq Ce^{-t}$ for 
$k\geq1,\ t\geq1$ with $C:=\sup_{k\geq1}k^2e^{-(2k-1)}<\infty.$
Therefore,
\begin{equation}\label{eq:ray-hessian-decay}
\frac{\dd^2}{\dd t^2}\log\mathcal K(y_0+tv)
\leq
C_0e^{-t}
\end{equation}
for $ t\geq1$ with
$C_0 := \frac{4C}{A_0}\sum_{k\geq1}A_k<\infty.$

Choose \(z(0)\in\Omega\cap(\mathbb C^*)^n\) with
\(\Log z(0)=y_0\), and define the ray $\Log z(t)=y_0+tv\in L$ for \(t\geq0\). Then
 $z_j(t):=z_j(0)e^{v_jt}$ is the curve in $\Omega$. 
  By
Corollary~\ref{lem:radial-formula} and
\eqref{eq:ray-hessian-decay},
\[ |\dot z(t)|_{\omega_\Omega}^2 = \frac14 \frac{\dd^2}{\dd t^2} \log\mathcal K(y_0+tv) \leq \frac{C_0}{4}e^{-t}. \]
Hence
\[ \operatorname{Length}_{\omega_\Omega} \{z(t):t\geq1\} \leq \frac{\sqrt{C_0}}{2} \int_1^\infty e^{-t/2}\,\dd t <\infty. \]
It follows that \(\{z(t_m)\}\) is a Cauchy sequence with respect to the Bergman metric for any \(t_m\to\infty\).
Suppose that $\omega_\Omega$ is complete. Then $z(t_m) \to z_\infty \in \Omega$ as \(t_m\to\infty\).
On the other hand,  since \(v_j>0\) for some \(j\), 
the curve \(z(t)\)  cannot be contained in any compact subset of \(\Omega\).
This contradiction implies that 
 \(\omega_\Omega\) is incomplete.
 \end{proof}

Define
$$ \widetilde\Omega := \operatorname{Exp}^{-1} \bigl(\Omega\cap(\mathbb C^*)^n\bigr) = L+i\mathbb R^n. $$
 Then
$ \operatorname{Exp}:\widetilde\Omega \rightarrow \Omega\cap(\mathbb C^*)^n$
is a holomorphic covering map with deck group
\(2\pi i\mathbb Z^n\).
Let
$$ I=\{j:\Omega\cap\{z_j=0\}\neq\varnothing\}, \quad J=\{1,\ldots,n\}\setminus I,$$
and define
$$ \widehat\Omega = \left\{ (u,w)\in\mathbb C^I\times\mathbb C^J: \Pi(u,w)\in\Omega \right\},$$
where
$$ \Pi(u,w) = \bigl((u_j)_{j\in I},(e^{w_k})_{k\in J}\bigr).$$

Since $\Pi:\widehat\Omega\rightarrow\Omega$
is a holomorphic covering map, the Bergman metric \(\omega_\Omega\) lifts
to the K\"ahler metric $\widehat \omega:=\Pi^*\omega_\Omega$
on each connected component of \(\widehat\Omega\).  Note that $\widehat \omega$ may be different from the intrinsic Bergman
metric of \(\widehat\Omega\).  Since \(\Pi\) is a covering, 
$\widehat \omega$ is a complete K\"ahler-Einstein metric with Einstein constant $-1$. 
It follows from the uniqueness that 
any complete K\"ahler--Einstein metric on $ \widehat\Omega$ with Einstein constant $-1$ must be $\widehat \omega$. 
The affine lifting principle used below is standard in the theory of
Reinhardt domains (cf. \cite{FIK96a, JP08}), while the precise
extension statement does not
seem to be available in this form, so we include the argument here.  

\begin{lemma}\label{lem:toric-extension}
Let
\begin{equation}\label{eq:affine-log-auto}
 \phi(y)=By+d
\end{equation}
be an affine automorphism of \(L\), and suppose that
\begin{equation}\label{eq:Bfixes}
 Be_j=e_j\quad {\rm for~all}\quad j\in I.
\end{equation}
Choose \(c\in\mathbb C^n\) such that \(\operatorname{Re}c=d\).  Then
\[
 \widetilde\phi(W)=BW+c
\]
defines an affine biholomorphism
of \(\widetilde\Omega\). Moreover, $ \widetilde\phi$ induces, on the
dense torus locus, a holomorphic map of \(\widehat\Omega\), which extends to a biholomorphism between connected components of
\(\widehat\Omega\) and is an isometry for $\widehat \omega$.
\end{lemma}

\begin{proof}
By re-ordering the coordinates as $\{I, J\}$, 
 it follows from \eqref{eq:Bfixes} that 
$ B=\begin{pmatrix}I_{|I|}&C\\0&D\end{pmatrix},$
for $ D\in\operatorname{GL}(|J|,\R).$
Choose $c\in\C^n$ with $\operatorname{Re}c=d$.  On the dense locus
where all $u_j\neq0$, the lift is induced by $W\mapsto BW+c$.  It
extends to the ambient space $\C^I\times\C^J$ by
\begin{align*}
 u'_j&=e^{c_j}u_j
 \exp\left(\sum_{k\in J}B_{jk}w_k\right),
 &&j\in I, \\ 
 w'_k&=\sum_{\ell\in J}B_{k\ell}w_\ell+c_k,
 &&k\in J.
\end{align*}
The inverse matrix $B^{-1} =
\begin{pmatrix}
I_{|I|}&-CD^{-1}\\
0&D^{-1}
\end{pmatrix}$
satisfies \(B^{-1}e_j=e_j\) for every \(j\in I\).
Since \(B\) is real, $\operatorname{Re}\widetilde\Phi(W)=B\operatorname{Re}W+d=\phi(\operatorname{Re}W).$
Because \(\phi\) is an affine automorphism of \(L\),
\(\widetilde\Phi\) is a biholomorphic automorphism of
\(\widetilde\Omega\).
On the locus where \(u_j\neq0\) for every \(j\in I\), choose 
$W_j=\log u_j$ for $ j\in I$
and $W_k=w_k$ for $ k\in J.$
The block form of \(B\) yields
\begin{align*}
W'_j
&=
W_j+\sum_{k\in J}B_{jk}w_k+c_j,
&&j\in I,\\
W'_k
&=
\sum_{\ell\in J}B_{k\ell}w_\ell+c_k,
&&k\in J.
\end{align*}
By taking exponential, we obtain
\begin{align}
u'_j
&=
e^{c_j}u_j
\exp\left(\sum_{k\in J}B_{jk}w_k\right),
&&j\in I,
\label{eq:toric-u-map}\\
w'_k
&=
\sum_{\ell\in J}B_{k\ell}w_\ell+c_k,
&&k\in J.
\label{eq:toric-w-map}
\end{align}
Note that these formulas are independent of the choices of
\(\log u_j\).  
Moreover, the right-hand sides of
\eqref{eq:toric-u-map}--\eqref{eq:toric-w-map} are holomorphic for all
$(u,w)\in\mathbb C^I\times\mathbb C^J,$
including points where some \(u_j=0\).  We therefore obtain an ambient
holomorphic map
\[
\widehat\Phi:
\mathbb C^I\times\mathbb C^J
\rightarrow
\mathbb C^I\times\mathbb C^J.
\]
Since \(B^{-1}\) has the same block form, the affine inverse $W\mapsto B^{-1}(W-c)$
induces an inverse map of the same form.  Hence
\(\widehat\Phi\) is a biholomorphism.

Set
\[
\mathcal E
:=
\bigcup_{j\in I}\{u_j=0\}.
\]
This is the union of the coordinate divisors in the partial
logarithmic cover.  Since the exponential factors in
\eqref{eq:toric-u-map} are nowhere zero, we have
$u'_j=0$ if and only if $u_j=0$ for all  $j\in I$.
Therefore,
\begin{equation}\label{eq:divisor-preserved}
\widehat\Phi^{-1}(\mathcal E)
=
\mathcal E
=
\widehat\Phi(\mathcal E).
\end{equation}
We first prove that \(\widehat\Phi\) preserves
\(\widehat\Omega\) away from \(\mathcal E\).  Let
$p=(u,w)\in\widehat\Omega\setminus\mathcal E.$
Choose \(W_j=\log u_j\) for \(j\in I\), and set
\(W_k=w_k\) for \(k\in J\).  Since
\(\Pi(p)\in\Omega\cap(\mathbb C^*)^n\), we have
$W\in\widetilde\Omega.$ 
\eqref{eq:toric-u-map}--\eqref{eq:toric-w-map} are
induced by \(W\mapsto BW+c\).  Since
\(\widetilde\Phi(\widetilde\Omega)=\widetilde\Omega\), it follows that
$\widehat\Phi(p)\in\widehat\Omega.$ 
By \eqref{eq:divisor-preserved}, the image is also outside
\(\mathcal E\).  Thus
\begin{equation}\label{eq:torus-locus-preserved}
\widehat\Phi
\bigl(\widehat\Omega\setminus\mathcal E\bigr)
\subset
\widehat\Omega\setminus\mathcal E.
\end{equation}
We next prove that points of
\(\widehat\Omega\cap\mathcal E\) are also mapped into
\(\widehat\Omega\).  Let
$p\in\widehat\Omega\cap\mathcal E$
and let $q:=\widehat\Phi(p).$
Choose a sufficiently small open ball \(\mathcal U\) centered at \(p\)
such that
$\overline{\mathcal U}\subset\widehat\Omega.$
Since \(\widehat\Phi\) is an ambient biholomorphism,
$\mathcal V:=\widehat\Phi(\mathcal U)$
is an open neighborhood of \(q\).
Let \(q'\in\mathcal V\setminus\mathcal E\) and write
$p':=\widehat\Phi^{-1}(q').$
Then \(p'\in\mathcal U\subset\widehat\Omega\).  Moreover,
\eqref{eq:divisor-preserved} implies that \(p'\notin\mathcal E\).
Hence \eqref{eq:torus-locus-preserved} yields 
$q'=\widehat\Phi(p')\in\widehat\Omega.$
This shows \begin{equation}\label{eq:image-minus-divisors}
\mathcal V\setminus\mathcal E
\subset\widehat\Omega.
\end{equation}
The set \(\mathcal E\) is a finite union of complex hyperplanes and
therefore has empty interior.  Consequently,
\(\mathcal V\setminus\mathcal E\) is dense in \(\mathcal V\).
Since $\mathcal V\setminus\mathcal E\subset\widehat\Omega,$
we obtain $q\in\overline{\widehat\Omega}.$
Suppose, for contradiction, that \(q\notin\widehat\Omega\).  Then
$q\in\partial\widehat\Omega.$
Because \(\Pi\) is a local biholomorphism and \(\partial\Omega\) is
smooth, $\partial\widehat\Omega=\Pi^{-1}(\partial\Omega)$
is a smooth real hypersurface.  Hence there are a neighborhood
\(\mathcal W\) of \(q\) and a smooth defining function \(\rho\) such
that $\widehat\Omega\cap\mathcal W=\{\rho<0\},
\partial\widehat\Omega\cap\mathcal W=\{\rho=0\},
d\rho\neq0$ hold
on \(\partial\widehat\Omega\cap\mathcal W\).
Since \(q\in\partial\widehat\Omega\), the exterior side
$\mathcal W\cap\{\rho>0\}$
meets every sufficiently small neighborhood of \(q\).  Therefore,
$\mathcal V\cap\mathcal W\cap\{\rho>0\}$
is a nonempty open set.  As \(\mathcal E\) has empty interior, we can
choose $q'\in \mathcal V\cap\mathcal W\cap\{\rho>0\} \setminus\mathcal E.$
Since \(q'\in\mathcal V\setminus\mathcal E\),
\eqref{eq:image-minus-divisors} implies
$q'\in\widehat\Omega.$
On the other hand, \(\rho(q')>0\), so
\(q'\notin\widehat\Omega\), a contradiction.  Hence
$ q=\widehat\Phi(p)\in\widehat\Omega. $
We have thus proved
$ \widehat\Phi(\widehat\Omega)\subset\widehat\Omega.$
Since the inverse map is induced by $W\mapsto B^{-1}(W-c)$, 
 the same argument applying 
to \(\widehat\Phi^{-1}\), yields
$\widehat\Phi^{-1}(\widehat\Omega) \subset\widehat\Omega.$
It follows that $\widehat\Phi(\widehat\Omega)=\widehat\Omega.$

Because \(\widehat\Phi\) is a homeomorphism of
\(\widehat\Omega\), it permutes its connected components.  If
\(\widehat\Omega_1\) is a connected component, then
$\widehat\Omega_2:=\widehat\Phi(\widehat\Omega_1)$
is another connected component, and
$ \widehat\Phi: \widehat\Omega_1 \rightarrow \widehat\Omega_2$
is a biholomorphism.
Let $\Pi_i:=\Pi|_{\widehat\Omega_i},
\widehat\omega_i:=\Pi_i^*\omega_\Omega,$ for 
$ i=1,2.$
Note that \(\widehat\omega_i\) is a complete
K\"ahler-Einstein metric with Einstein constant -1. 
The metric $\omega':=\widehat\Phi^*\widehat\omega_2$ on \(\widehat\Omega_1\) is 
also a complete
K\"ahler-Einstein metric with Einstein constant -1. 
The uniqueness thus yields $-\widehat\omega_1=-\omega'$ \cite{CY80}. Therefore \(\widehat\Phi\) is an isometry between the connected components. 
\end{proof}

\begin{lemma}
\label{lem:spectrum-rigidity}
Under the hypotheses of Lemma~\ref{lem:toric-extension}, there exist
$\ell\in\R^n$ and $c_0\in\R$ such that
\begin{equation}\label{eq:F-affine-difference}
 \log\mathcal K(By+d)-\log\mathcal K(y)=2\ell\cdot y+c_0,
\end{equation}
and
\begin{equation}\label{eq:difference-spectrum-invariance}
 B^T(\mathcal P-\mathcal P)=\mathcal P-\mathcal P.
\end{equation}
\end{lemma}

\begin{proof}
Let \(W=y+i\theta\). 
 The pullback of $\omega_\Omega$ to $ \widetilde\Omega$ can be written as 
$\widetilde\omega=i\partial\bar\partial  \log\mathcal K(\operatorname{Re}W).$
It follows from  Lemma~\ref{lem:toric-extension}
that $\widetilde\phi^*\widetilde\omega =\widetilde\omega$. 
Consequently, we have 
\begin{equation}\label{eq:pluriharmonic-potential-difference}
\partial\bar\partial \left(  \log\mathcal K(B\operatorname{Re}W+d) -  \log\mathcal K(\operatorname{Re}W) \right) = 0.
\end{equation}
It follows from \eqref{eq:pluriharmonic-potential-difference} that
$D^2 \left( \log\mathcal K(By+d)- \log\mathcal K(y)\right)=0$
for all $y\in L$. 
Since  \(L\) is convex, the standard argument in convex geometry yields that 
there exist
\(\ell \in\mathbb R^n\) and \(c_0\in\mathbb R\) such that
$\log\mathcal K(By+d)- \log\mathcal K(y)=2\ell \cdot y+c_0.$

Since
\begin{align*}
\mathcal K(By+d)=\sum_{\alpha\in\mathcal P} c_\alpha e^{2\alpha\cdot(By+d)}=\sum_{\alpha\in\mathcal P} c_\alpha e^{2\alpha\cdot d} e^{2\alpha\cdot By}=
\sum_{\alpha\in\mathcal P} c_\alpha e^{2\alpha\cdot d} e^{2(B^T\alpha)\cdot y},
\end{align*}
and
\begin{align*}
e^{2\ell\cdot y+c_0}\mathcal K(y)= e^{c_0}e^{2\ell\cdot y} \sum_{\alpha\in\mathcal P} c_\alpha e^{2\alpha\cdot y}
= e^{c_0} \sum_{\alpha\in\mathcal P} c_\alpha e^{2(\alpha+\ell)\cdot y},
\end{align*}
it follows from  \eqref{eq:F-affine-difference} that  
\begin{equation}\label{eq:two-exponential-series}
\sum_{\alpha\in\mathcal P}
c_\alpha e^{2\alpha\cdot d}
e^{2(B^T\alpha)\cdot y}
=
e^{c_0}
\sum_{\alpha\in\mathcal P}
c_\alpha e^{2(\alpha+\ell)\cdot y}
\end{equation}
for all $y\in L.$
Define finite positive discrete measures on \(\mathbb R^n\) by
\begin{equation}\label{eq:left-spectral-measure}
\mu
:=
\sum_{\alpha\in\mathcal P}
c_\alpha
e^{2\alpha\cdot d+2(B^T\alpha)\cdot y_0}
\delta_{B^T\alpha}
\end{equation}
and
\begin{equation}\label{eq:right-spectral-measure}
\nu
:=
e^{c_0}
\sum_{\alpha\in\mathcal P}
c_\alpha
e^{2(\alpha+\ell)\cdot y_0}
\delta_{\alpha+\ell},
\end{equation}
where \(\delta_x\) denotes the Dirac measure at \(x\).
It follows that
\begin{align*}
\mu(\mathbb R^n)
=\sum_{\alpha\in\mathcal P} c_\alpha e^{2\alpha\cdot d+2(B^T\alpha)\cdot y_0}=
\mathcal K(By_0+d)
<\infty,
\end{align*}
\begin{align*}
\nu(\mathbb R^n) = e^{c_0} \sum_{\alpha\in\mathcal P} c_\alpha e^{2(\alpha+\ell)\cdot y_0}=
e^{c_0+2\ell\cdot y_0}\mathcal K(y_0) <\infty
\end{align*}
and $\mu(\mathbb R^n)= \nu(\mathbb R^n) $ by (\ref{eq:F-affine-difference}).
Fix \(y_0\in L\).  Since \(L\) is open, there exists an open
neighborhood \(U\) of \(0\) in \(\mathbb R^n\) such that
$y_0+h\in L$ for  $h\in U$.
Because \(y\mapsto By+d\) is an automorphism of \(L\), we also have
$B(y_0+h)+d\in L$
It follows from substituting \(y=y_0+h\) into
\eqref{eq:two-exponential-series} that
\begin{equation}\label{eq:laplace-transform-equality}
\int_{\mathbb R^n}e^{2x\cdot h}\,\dd\mu(x)
=
\int_{\mathbb R^n}e^{2x\cdot h}\,\dd\nu(x),
\end{equation}
for $h\in U.$
Thus \(\mu\) and \(\nu\) have the same  Laplace transform
in a neighborhood of the origin.
 For $\zeta\in \mathcal T := \{\zeta\in\mathbb C^n:   \operatorname{Re}\zeta\in U\},$
\[
M_\mu(\zeta)
:=
\int_{\mathbb R^n}
e^{2\zeta\cdot x}\,\dd\mu(x),
\quad
M_\nu(\zeta)
:=
\int_{\mathbb R^n}
e^{2\zeta\cdot x}\,\dd\nu(x).
\]
are 
holomorphic functions on a possibly small tube neighborhood of
\(i\mathbb R^n\).  By
\eqref{eq:laplace-transform-equality} and the uniqueness of holomorphic functions, 
we obtain
$M_\mu(\zeta)=M_\nu(\zeta)$. 
 In particular, for every
\(\xi\in\mathbb R^n\), we may take
\(\zeta=i\xi/2\) and obtain
\[
\int_{\mathbb R^n}e^{i\xi\cdot x}\,\dd\mu(x)
=
\int_{\mathbb R^n}e^{i\xi\cdot x}\,\dd\nu(x).
\]
Thus \(\mu\) and \(\nu\) have the same Fourier transform.  The
uniqueness theorem for Fourier transforms of finite measures implies
$\mu=\nu.$

Note that all coefficients in
\eqref{eq:left-spectral-measure} and
\eqref{eq:right-spectral-measure} are strictly positive.  Moreover,
\(B\) is invertible because \(y\mapsto By+d\) is an affine
automorphism.  Since \(\mathcal P\subset\mathbb Z^n\), the sets
$B^T\mathcal P$ and $\mathcal P+\ell$
are discrete subsets of \(\mathbb R^n\).  It follows that
$ \operatorname{supp}\mu=B^T\mathcal P, 
\operatorname{supp}\nu=\mathcal P+\ell.$
The equality \(\mu=\nu\) therefore yields
$B^T\mathcal P=\mathcal P+\ell$
and (\ref{eq:difference-spectrum-invariance}) thus
 follows.
\end{proof}

\subsection{Proof of Theorem \ref{unbounded}}

\begin{proof}[Proof of Theorem \ref{unbounded}]
We follow the idea in the proof of Theorem \ref{bounded}. 
By Proposition \ref{prop:connected} and Proposition \ref{prop:sphere}, $M$ is a connected closed strongly pseudoconvex
spherical tube hypersurface.
 After an invertible real affine
transformation (\ref{eq:affine-change}), 
the hypersurface $A^{-1}(S-b)$ is one of three  models in Theorem \ref{DY} by \cite{DY85}.

We first consider Type I model. Let 
\begin{equation}\label{eq:q-Type-I-components}
C_{\rm I}
=\left\{q\in\mathbb R^n:
\begin{array}{ll}
q_0\geq0, & \\[2pt]
q_\alpha\leq0,
    & 1\leq\alpha\leq m,\\[2pt]
q_\alpha=0,
    & m<\alpha\leq n-1
\end{array}
\right\}.
\end{equation}

We first show that
\(C_{\mathrm I}\subset \operatorname{Rec}(L_{\mathrm I}^{(m)})\). Let $q \in C_{\mathrm I}$. 
Fix \(x\in L_{\mathrm I}^{(m)}\) and \(t\geq0\).  Since
\(q_\alpha\leq0\) for \(1\leq\alpha\leq m\), we have
$e^{x_\alpha+tq_\alpha}\leq e^{x_\alpha}$
for $1\leq\alpha\leq m.$
Moreover, \(q_\alpha=0\) for \(m<\alpha\leq n-1\), and hence
$(x_\alpha+tq_\alpha)^2=x_\alpha^2$
for $m<\alpha\leq n-1.$
Finally, \(q_0\geq0\), so \(x_0+tq_0\geq x_0\).  Therefore,
\begin{align*}
x_0+tq_0
&\geq x_0 >
\sum_{\alpha=1}^{m}e^{x_\alpha}
+
\sum_{\alpha=m+1}^{n-1}x_\alpha^2 \geq
\sum_{\alpha=1}^{m}e^{x_\alpha+tq_\alpha}
+
\sum_{\alpha=m+1}^{n-1}
   (x_\alpha+tq_\alpha)^2.
\end{align*}
Thus \(x+tq\in L_{\mathrm I}^{(m)}\) for every
\(x\in L_{\mathrm I}^{(m)}\) and every \(t\geq0\).  Consequently,
$C_{\mathrm I} \subset \operatorname{Rec}\bigl(L_{\mathrm I}^{(m)}\bigr).$
For the reverse inclusion, suppose that
$q\in\operatorname{Rec}\bigl(L_{\mathrm I}^{(m)}\bigr).$
Fix any \(x\in L_{\mathrm I}^{(m)}\).  By definition,
\begin{equation}\label{eq:type-I-ray-inequality}
x_0+tq_0
>
\sum_{\alpha=1}^{m}e^{x_\alpha+tq_\alpha}
+
\sum_{\alpha=m+1}^{n-1}
   (x_\alpha+tq_\alpha)^2
\end{equation}
for every \(t\geq0\).
We first prove that \(q_0\geq0\).  If \(q_0<0\), then
$x_0+tq_0\rightarrow-\infty$
as $t\rightarrow+\infty.$
On the other hand, the right-hand side of
\eqref{eq:type-I-ray-inequality} is nonnegative for every \(t\).
Hence \eqref{eq:type-I-ray-inequality} fails for all sufficiently
large \(t\), which is a contradiction.  Therefore,
$q_0\geq0.$
Next, we prove 
$q_\alpha\leq0$ for $ 1\leq\alpha\leq m.$
Suppose, to the contrary, that \(q_k>0\) for some
\(k\in\{1,\ldots,m\}\).  From
\eqref{eq:type-I-ray-inequality}, after discarding all the other
nonnegative terms on the right, we obtain
$x_0+tq_0>e^{x_k+tq_k}$ for  $t\geq0$.
The left-hand side grows at most linearly in \(t\), whereas
$e^{x_k+tq_k}=e^{x_k}e^{tq_k}$
grows exponentially because \(q_k>0\).  More precisely,
$\frac{e^{x_k+tq_k}}{1+t}\rightarrow+\infty$
as $t\rightarrow+\infty.$
Thus, for all sufficiently large \(t\),
$e^{x_k+tq_k}>x_0+tq_0,$
contradicting \eqref{eq:type-I-ray-inequality}.  Hence
\(q_\alpha\leq0\) for every \(1\leq\alpha\leq m\).
Finally, we prove that
$q_\alpha=0$ for $ m<\alpha\leq n-1.$
Suppose that \(q_k\neq0\) for some
\(k\in\{m+1,\ldots,n-1\}\).  Again,
\eqref{eq:type-I-ray-inequality} implies
$x_0+tq_0>(x_k+tq_k)^2$ for $t\geq0$.
But $(x_k+tq_k)^2=q_k^2t^2+2x_kq_kt+x_k^2$
grows quadratically in \(t\), whereas \(x_0+tq_0\) grows at most
linearly.  In particular,
$(x_k+tq_k)^2-(x_0+tq_0) \rightarrow+\infty$ 
as $t\rightarrow+\infty.$
This contradicts the preceding strict inequality.  Therefore,
\(q_k=0\).  Since \(k\) was arbitrary,
$q_\alpha=0$ for $ m<\alpha\leq n-1.$
 Hence
$\operatorname{Rec}\bigl(L_{\mathrm I}^{(m)}\bigr)
\subset C_{\mathrm I}.$
Combining the two inclusions yields $C_{\mathrm I} = \operatorname{Rec}(L_{\mathrm I}^{(m)})$.

For every \(j\in I\), Lemma~\ref{lem:recession} and the affine
relation $L=A\bigl(L_{\mathrm I}^{(m)}\bigr)+b$
imply  $-e_j\in\operatorname{Rec}(L)=A\operatorname{Rec}\bigl(L_{\mathrm I}^{(m)}\bigr).$
Applying \(A^{-1}\), we obtain
\begin{equation}\label{eq:qfilled}
q^{(j)}
:=
-A^{-1}e_j
\in
\operatorname{Rec}\bigl(L_{\mathrm I}^{(m)}\bigr)
=
C_{\mathrm I}
\end{equation}
for $ j\in I.$

\begin{lemma}\label{lem:vertical-origin}
If $q^{(j)}_0>0$ for some $j\in I$, then $0\in\Omega$.
\end{lemma}

\begin{proof}
Fix \(j\in I\) such that
$q^{(j)}_0>0$. 
For simplicity, write $q=q^{(j)}.$
By \eqref{eq:qfilled}, \(q\in C_{\mathrm I}\).  
Consider a point
$p=(p_1,\ldots,p_n)\in H_j =\{z\in\mathbb C^n:z_j=0\}$
such that $p_k\neq0$ 
for $k\neq j$. 
Thus $\log|p_k|$ are finite.
Choose \(y^*\in\mathbb R^n\) by setting
$y_k^*=\log|p_k|$ for $ k\neq j$
and choosing \(y_j^*\in\mathbb R\) arbitrarily.  Define
$x^*:=A^{-1}(y^*-b).$
For \(t\geq0\), set
$y(t):=y^*-te_j.$
Using \(q=-A^{-1}e_j\), we obtain
\begin{align*}
A^{-1}(y(t)-b)
=A^{-1}(y^*-b)-tA^{-1}e_j=x^*+tq.
\end{align*}
Let $\Psi(x):=x_0-\sum_{\alpha=1}^{m}e^{x_\alpha}-\sum_{\alpha=m+1}^{n-1}x_\alpha^2$
so that $L_{\mathrm I}^{(m)}=\{x\in\mathbb R^n:\Psi(x)>0\}.$
Using \eqref{eq:q-Type-I-components}, we have
$e^{x_\alpha^*+tq_\alpha} \leq e^{x_\alpha^*}$ for
$ 1\leq\alpha\leq m$
and
$(x_\alpha^*+tq_\alpha)^2=(x_\alpha^*)^2$
for $ m<\alpha\leq n-1.$
Consequently,
\begin{align*}
\Psi(x^*+tq)
&=x_0^*+tq_0-\sum_{\alpha=1}^{m}e^{x_\alpha^*+tq_\alpha}-\sum_{\alpha=m+1}^{n-1} (x_\alpha^*+tq_\alpha)^2\\
&\geq x_0^*+tq_0-\sum_{\alpha=1}^{m}e^{x_\alpha^*}-\sum_{\alpha=m+1}^{n-1}(x_\alpha^*)^2.
\end{align*}
Since \(q_0>0\), the last expression tends to \(+\infty\) as
\(t\to+\infty\).  Hence there exists \(T>0\) such that
$\Psi(x^*+tq)>0$ for
$t\geq T.$
Therefore,
$x^*+tq\in L_{\mathrm I}^{(m)}$ for
$t\geq T$
Since \(y=Ax+b\) and
$L=A\bigl(L_{\mathrm I}^{(m)}\bigr)+b,$
it follows that
$y(t)=A(x^*+tq)+b\in L$ for 
$t\geq T$.
For \(k\neq j\), set
$z_k(t):=p_k,$
while for the \(j\)-th coordinate choose any fixed
\(\theta_j\in\mathbb R\) and set
$z_j(t):=e^{y_j^*-t+i\theta_j}.$
Then
$\log|z(t)|=y(t)\in L$ for 
$t\geq T$. 
Because \(\Omega\) is Reinhardt, it follows that
$z(t)\in\Omega$ for 
$t\geq T$ and 
$z(t)\rightarrow p$ 
as $t\to+\infty.$
This proves that
$p\in\overline\Omega$ 
whenever \(p\in H_j\) and \(p_k\neq0\) for every \(k\neq j\).
The set
$H_j^\circ:=\{p\in H_j:p_k\neq0\text{ for every }k\neq j\}$
is dense in \(H_j\).  Since \(\overline\Omega\) is closed and
\(H_j^\circ\subset\overline\Omega\), we conclude that
\begin{equation}\label{eq:hyperplane-in-closure}
H_j\subset\overline\Omega.
\end{equation}

Suppose, to the contrary, that there exists
$p\in H_j\cap\partial\Omega.$ 
Choose a smooth torus-invariant defining function \(\rho\) for
\(\Omega\) in a neighborhood \(U\) of \(p\), normalized so that
\[
\Omega\cap U=\{\rho<0\},
\quad
\partial\Omega\cap U=\{\rho=0\},
\quad
d\rho\neq0\quad\text{on }\partial\Omega\cap U.
\]
By \eqref{eq:hyperplane-in-closure},
$H_j\cap U\subset\overline\Omega.$
It follows that
$\rho\leq0$ on $H_j\cap U.$
Since \(p\in\partial\Omega\), we have \(\rho(p)=0\).  Hence the
restriction
$\rho|_{H_j\cap U}$
has a local maximum at \(p\).  Therefore,
$d\rho(p)[v]=0$
for every $v\in T_pH_j.$
It follows from the torus invariance that  $d\rho(p)=0.$ This is a contradiction. 
Consequently,
$H_j\cap\partial\Omega=\varnothing$ and thus 
$H_j\subset\Omega.$
In particular, 
$0\in\Omega.$
\end{proof}

By Theorem \ref{thm:origin-case} and Lemma~\ref{lem:vertical-origin},
\begin{equation}\label{eq:q0zero}
 q^{(j)}_0=0
\end{equation}
for $j\in I$. 
We now use a one-parameter family of affine automorphisms of the Type I
model.  For \(s>0\), define
\begin{equation}\label{eq:type-I-dilation}
\delta_s(x_0,x_1,\ldots,x_{n-1})
=
\bigl(
s^2x_0,\,
x_1+2\log s,\ldots,x_m+2\log s,\,
sx_{m+1},\ldots,sx_{n-1}
\bigr).
\end{equation}
Recall that
\[
L_{\mathrm I}^{(m)}
=
\left\{
x\in\mathbb R^n:
x_0>
\sum_{\alpha=1}^{m}e^{x_\alpha}
+
\sum_{\alpha=m+1}^{n-1}x_\alpha^2
\right\}.
\]
Under \(\delta_s\), the right-hand side of the defining inequality
transforms as
\begin{align*}
\sum_{\alpha=1}^{m}e^{x_\alpha+2\log s}+
\sum_{\alpha=m+1}^{n-1}(sx_\alpha)^2
=s^2\sum_{\alpha=1}^{m}e^{x_\alpha}+s^2\sum_{\alpha=m+1}^{n-1}x_\alpha^2
=s^2\left(\sum_{\alpha=1}^{m}e^{x_\alpha}+\sum_{\alpha=m+1}^{n-1}x_\alpha^2\right).
\end{align*}
Since \(s^2>0\), we have
$x_0>\sum_{\alpha=1}^{m}e^{x_\alpha}+
\sum_{\alpha=m+1}^{n-1}x_\alpha^2$
if and only if 
$s^2x_0 >\sum_{\alpha=1}^{m}e^{x_\alpha+2\log s}+\sum_{\alpha=m+1}^{n-1}(sx_\alpha)^2.$
Thus \(\delta_s\) preserves \(L_{\mathrm I}^{(m)}\).  Moreover,
\(\delta_{1/s}\) is its inverse, so \(\delta_s\) is an affine
automorphism of \(L_{\mathrm I}^{(m)}\).
Write
$\delta_s(x)=D_sx+a_s,$
where $D_s=\operatorname{diag}\bigl( s^2,I_m,sI_{n-1-m} \bigr)$
and
$a_s=\bigl( 0,2\log s,\ldots,2\log s,0,\ldots,0 \bigr)$, 
where the entries \(2\log s\) occur in positions \(1,\ldots,m\).
Since \(q^{(j)}\in C_{\mathrm I}\), $q^{(j)}_\alpha=0$ for
$m<\alpha\leq n-1.$
Together with \eqref{eq:q0zero}, this shows that
$q^{(j)} = \bigl( 0,q^{(j)}_1,\ldots,q^{(j)}_m,0,\ldots,0 \bigr).$
The matrix \(D_s\) acts as the identity on the coordinates
\(1,\ldots,m\), while the coordinates on which it acts by \(s^2\) or
\(s\) vanish in \(q^{(j)}\).  Therefore we get 
\begin{equation}\label{eq:Ds-fixes-qj}
D_sq^{(j)}=q^{(j)}
\end{equation}
for $j\in I$.
Let $\eta(x):=Ax+b$ so that $L=\eta\bigl(L_{\mathrm I}^{(m)}\bigr).$
Conjugating \(\delta_s\) by \(\eta\), define
$\phi_s:=\eta\circ\delta_s\circ\eta^{-1}.$
Then \(\phi_s\) is an affine automorphism of \(L\).  More explicitly,
\begin{align*}
\phi_s(y)=A\delta_s\bigl(A^{-1}(y-b)\bigr)+b =AD_sA^{-1}y +\bigl(Aa_s+b-AD_sA^{-1}b\bigr).
\end{align*}
Thus the linear part of \(\phi_s\) is
$B_s:=AD_sA^{-1}.$ 
For 
$j\in I$, 
since
$q^{(j)}=-A^{-1}e_j,$
it follows from \eqref{eq:Ds-fixes-qj} that 
\begin{align*}
B_se_j=AD_sA^{-1}e_j=-AD_sq^{(j)}=-Aq^{(j)}=e_j.
\end{align*}
The affine automorphism \(\phi_s\) and its linear part \(B_s\)
therefore satisfy the hypotheses of
Lemma~\ref{lem:toric-extension}.  Applying
Lemma~\ref{lem:spectrum-rigidity}, we obtain
\begin{equation}\label{eq:Bs-spectrum}
B_s^T(\mathcal P-\mathcal P) =\mathcal P-\mathcal P
\end{equation}
for $s>0$.
Fix \(d\in\mathcal P-\mathcal P\).  Since \(D_s\) depends
continuously on \(s\), so do \(B_s=AD_sA^{-1}\) and the map from
$(0,\infty)$ to $\mathbb R^n$ given by 
$s\mapsto B_s^Td.$
By \eqref{eq:Bs-spectrum},
$B_s^Td\in\mathcal P-\mathcal P$ for 
$s>0$. 
Since $\mathcal P-\mathcal P\subset\mathbb Z^n$ 
is a discrete subset of
\(\mathbb R^n\)
every continuous
map from \((0,\infty)\) into this discrete set is constant.
Consequently,
$B_s^Td=B_1^Td$ for 
$s>0.$ 
At \(s=1\), we have
$D_1=I$
and hence $B_1=AD_1A^{-1}=I.$
It follows that
\begin{equation}\label{eq:Bs-fixes-spectrum-pointwise}
B_s^Td=d
\end{equation}
for for every $d\in\mathcal P-\mathcal P$
 and every $s>0.$
By Lemma~\ref{lem:spectrum-spans},
 \eqref{eq:Bs-fixes-spectrum-pointwise} 
and linearity that $B_s^T=I,$ 
and hence 
$B_s=I$
for every \(s>0\).
On the other hand, 
$D_s=A^{-1}B_sA=I.$ 
This is impossible when \(s\neq1\), since
$D_s = \operatorname{diag}\bigl( s^2,I_m,sI_{n-1-m} \bigr)$
has first diagonal entry \(s^2\neq1\).  
This 
contradiction implies that the Type I model cannot occur.

\medskip

Recall  the recession cone of the Type II model
\[
L_{\mathrm{II}}
=
\left\{
x=(x_0,x_1,\ldots,x_{n-1})\in\mathbb R^n:
0<x_0<\pi,\quad
\sum_{\alpha=1}^{n-1}e^{x_\alpha}<\sin x_0
\right\}
\]
 is given by 
$C_{\mathrm{II}} :=\operatorname{Rec}(L_{\mathrm{II}})=\{0\}\times\mathbb R_-^{n-1}=\left\{q\in\mathbb R^n: q_0=0, q_\alpha\leq0\ \text{for }1\leq\alpha\leq n-1\right\}$
from the proof of Theorem \ref{bounded}.
It follows that 
\begin{equation}\label{eq:rec-L-Type-II}
\operatorname{Rec}(L)=A C_{\mathrm{II}}.
\end{equation}
Consider the affine reflection
$\sigma(x_0,x') = (\pi-x_0,x')$
with $x'=(x_1,\ldots,x_{n-1}).$
It is easy to verify that $\sigma(L_{\mathrm{II}})=L_{\mathrm{II}}.$
Moreover, \(\sigma^2=I\), so \(\sigma\) is an affine automorphism of
\(L_{\mathrm{II}}\).
The linear part of \(\sigma\) is
$R=\operatorname{diag}(-1,I_{n-1})$ and $Rq=q$ for $q\in C_{\mathrm{II}}$. 
Define $\phi(y):=A\sigma\bigl(A^{-1}(y-b)\bigr)+b.$ 
Then \(\phi\) is an affine automorphism of \(L\), and the linear part is
$B:=ARA^{-1}.$
By Lemma~\ref{lem:recession} and  \eqref{eq:rec-L-Type-II},
$-A^{-1}e_j\in C_{\mathrm{II}}.$ 
Setting $q^{(j)}:=-A^{-1}e_j$, then 
$Rq^{(j)}=q^{(j)}.$
Consequently,
\begin{align*}
Be_j =ARA^{-1}e_j=-ARq^{(j)}= -Aq^{(j)}=e_j.
\end{align*}
It follows from Lemma~\ref{lem:spectrum-rigidity}
that 
\begin{equation}\label{eq:B-reflection-spectrum}
B^T(\mathcal P-\mathcal P)
=
\mathcal P-\mathcal P.
\end{equation}
By
Lemma~\ref{lem:spectrum-spans},
$\operatorname{span}_{\mathbb R}(\mathcal P-\mathcal P) =\mathbb R^n.$
Hence we may choose linearly independent vectors
$d_1,\ldots,d_n\in\mathcal P-\mathcal P.$
Let $D:=[d_1\ \cdots\ d_n].$
Since the \(d_i\) are linearly independent, \(D\) is invertible.
Moreover,
$d_i\in\mathcal P-\mathcal P\subset\mathbb Z^n,$
so \(D\) is an integer matrix with nonzero integer determinant  and 
$D^{-1}=\frac{1}{\det D}\operatorname{adj}(D)$
shows that $D^{-1}\in M_n(\mathbb Q).$
By \eqref{eq:B-reflection-spectrum},
$B^Td_i\in\mathcal P-\mathcal P\subset\mathbb Z^n.$ 
Thus the matrix $ E:=[B^Td_1\ \cdots\ B^Td_n]$
also has integer entries.  Since
$B^TD=E,$ 
we obtain
$$B^T =ED^{-1} = [B^Td_1\ \cdots\ B^Td_n] [d_1\ \cdots\ d_n]^{-1} \in\operatorname{GL}(n,\mathbb Q).$$
In particular,
$B\in\operatorname{GL}(n,\mathbb Q).$
Let $H:=\ker(B-I).$
Since \(B=ARA^{-1}\), we have
\begin{align*}
H=\ker(ARA^{-1}-I)=A\ker(R-I).
\end{align*}
It follows that
$H = A\{q_0=0\}.$
On the other hand,
$\operatorname{span}(C_{\mathrm{II}}) =\{q\in\mathbb R^n:q_0=0\}.$
Using \eqref{eq:rec-L-Type-II}, we therefore obtain
\[ \operatorname{span}\bigl(\operatorname{Rec}(L)\bigr)=
A\operatorname{span}(C_{\mathrm{II}})=A\{q_0=0\}=H=\ker(B-I).
\]
Because \(B\) has rational entries, 
the kernel \(H\) 
admits a basis consisting of vectors in \(\mathbb Q^n\).  Consequently,
\begin{equation}\label{eq:rational-points-dense}
H\cap\mathbb Q^n
\quad\text{is dense in }H.
\end{equation}

We now show that \(L_{\mathrm{II}}\) is contained in the sum of a bounded
set and \(C_{\mathrm{II}}\).  If \(x\in L_{\mathrm{II}}\), then
$0<x_0<\pi.$
Furthermore,
$ e^{x_\alpha} < \sum_{\beta=1}^{n-1}e^{x_\beta} < \sin x_0 \leq1$ 
implies 
$x_\alpha<0$
for $1\leq\alpha\leq n-1$.
Writing $x_\alpha=\max\{x_\alpha,-1\}+\min\{x_\alpha+1,0\}$ and 
defining
$u(x):=\bigl(x_0, \max\{x_1,-1\}, \ldots, \max\{x_{n-1},-1\} \bigr), c(x) := \bigl(0, \min\{x_1+1,0\}, \ldots, \min\{x_{n-1}+1,0\} \bigr)$, 
then $x=u(x)+c(x).$
Moreover,
$ u(x)\in K:=[0,\pi]\times[-1,0]^{n-1},$
which is bounded, while
$ c(x)\in \{0\}\times\mathbb R_-^{n-1} = C_{\mathrm{II}}.$
Therefore,
$ L_{\mathrm{II}} \subset K+C_{\mathrm{II}}.$
Applying the affine map \(x\mapsto Ax+b\),  
we obtain
\begin{align}\label{eq:L-bounded-plus-recession}
L = A L_{\mathrm{II}}+b \subset (AK+b)+AC_{\mathrm{II}}= K'+\operatorname{Rec}(L),
\end{align}
where
$ K':=AK+b$
is bounded.

Suppose that $q_j\leq0$ for every vector \(q\in\operatorname{Rec}(L)\) and every $j$.  
Since \(K'\) is bounded, for every \(j\) there is a constant \(M_j\)
such that
$ k_j\leq M_j$ for any $k\in K'.$
If \(y\in L\), then by \eqref{eq:L-bounded-plus-recession} we can write
$y=k+q$
for some \(k\in K'\) and \(q\in\operatorname{Rec}(L)\).  Hence
$y_j=k_j+q_j\leq k_j\leq M_j.$
It would follow that, for every
\(z\in\Omega\cap(\mathbb C^*)^n\),
$ |z_j| = e^{\log|z_j|} \leq e^{M_j}$ for $ 1\leq j\leq n.$ 
The torus locus \(\Omega\cap(\mathbb C^*)^n\) is dense in \(\Omega\),
so the same estimates hold throughout \(\Omega\).  Therefore
 \(\Omega\) is bounded.  This is a contradiction.
Consequently, there exist
$q^{(0)}\in\operatorname{Rec}(L)$
and an index \(j\) such that
\begin{equation}\label{eq:positive-recession-coordinate}
q^{(0)}_j>0.
\end{equation}

Choose a point
$r$ in the relative interior $ \operatorname{relint}\bigl(\operatorname{Rec}(L)\bigr)$ of the convex cone $\operatorname{Rec}(L)$.
For \(0<\varepsilon<1\), set
$q_\varepsilon := (1-\varepsilon)q^{(0)}+\varepsilon r.$
A standard property of relative interiors of convex sets yields
$q_\varepsilon \in \operatorname{relint}\bigl(\operatorname{Rec}(L)\bigr)$ for $0<\varepsilon<1$. 
Moreover, by \eqref{eq:positive-recession-coordinate},
$ (q_\varepsilon)_j = (1-\varepsilon)q^{(0)}_j+\varepsilon r_j>0$
for all sufficiently small \(\varepsilon>0\).  Hence there exists
$ q^* \in  \operatorname{relint}\bigl(\operatorname{Rec}(L)\bigr)$
such that
$ q^*_j>0.$
Since
$ \operatorname{span}\bigl(\operatorname{Rec}(L)\bigr)=H,$
the relative interior $ \operatorname{relint}\bigl(\operatorname{Rec}(L)\bigr)$ of \(\operatorname{Rec}(L)\) is open in \(H\).
Thus there exists a relatively open neighborhood \(V\subset H\) of
\(q^*\) such that
$V\subset\operatorname{Rec}(L)$ and 
$q_j>0\quad(q\in V).$
By \eqref{eq:rational-points-dense}, we may therefore choose
$q\in V\cap\mathbb Q^n.$
Then
$ q\in\operatorname{Rec}(L)\cap\mathbb Q^n$ and 
$q_j>0.$
It now follows from 
Lemma~\ref{lem:rational-ray} that the Bergman metric $\omega_\Omega$ of
\(\Omega\) is incomplete.  This yields that the Type II
model cannot occur.

\medskip

The proof for Type III is similar to that of Type II and is simpler. 
For Type III,
\[
L_{\rm III}
=
\left\{x\in\mathbb R^n:
\sum_{\alpha=0}^{n-1}e^{x_\alpha}<1\right\},
\quad
\Rec(L_{\rm III})=-\mathbb R_{\geq0}^n.
\]
Thus $\Rec(L)=A(-\mathbb R_{\geq0}^n)$,
which is full-dimensional. Since every \(x_\alpha<0\), the decomposition
\[
x_\alpha=\max\{x_\alpha,-1\}+\min\{x_\alpha+1,0\}
\]
shows that \(L_{\rm III}\), and thus \(L\), lies in a bounded set
plus its recession cone.
If every \(q\in\Rec(L)\) had nonpositive components in the original
logarithmic \(y\)-coordinate system, then all coordinates of \(L\)
would be bounded above, forcing \(\Omega\) to be bounded.
Thus some \(q\in\Rec(L)\) has a positive coordinate. Perturbing
toward an interior point preserves this positivity and produces an interior
recession direction. Since \(\mathbb Q^n\) is dense in \(\mathbb R^n\), we
obtain
$q\in\Rec(L)\cap\mathbb Q^n$ with $q_j>0$ for some $j$. 
By Lemma~\ref{lem:rational-ray}, $\omega_\Omega$ is incomplete. Therefore Type III
is impossible.

Combining all three types, we thus prove that unbounded
smooth Reinhardt domain in complex dimension at least two cannot admit 
a well-defined complete Bergman-Einstein  metric.
\end{proof}

\section*{Acknowledgments}

The author would like to thank Siqi Fu for helpful discussions.
The author was partially supported by Zhejiang Provincial Natural Science
Foundation of China (Grant No.~LQKWL26A0201).

\end{document}